\documentclass{tac}

\usepackage{amsmath}

\title{Controlled theories, categorification, and homotopification}

\address{Unaffiliated\\
 Chicago, Illinois, USA}

\author{Johnathon Taylor}

\usepackage{tikz}
\usepackage{tikz-cd}

\newtheoremrm{construction}{Construction}

\begin{document}

\maketitle

\begin{abstract}
In this paper, we introduce the notion of a \emph{controlled theory}, originally
developed in the author's thesis, as a structural tool for
the study of higher categorical algebra. We define a notion of deformation for
pros and controlled theories in a cartesian closed category. Furthermore, we
show that deformations of controlled theories naturally produce Lawvere
theories enriched over the same base category. We construct functorial
one-dimensional categorifications and homotopifications of controlled theories,
yielding Lawvere $2$-theories and Lawvere theories enriched in simplicial sets,
respectively. As an application, we obtain a new model for $\infty$-groups and
construct a model of coherent group-like $E_\infty$-spaces.
\end{abstract}

\copyrightyear{2026}


\eaddress{jt3theend17@gmail.com}

\tableofcontents

\section{Introduction}
Given a symmetric monoidal category $\mathcal{V}$ together with a category $C$ and an inclusion of categories $C\hookrightarrow\mathcal{V}$ that identifies the objects of $C$ as objects of $\mathcal{V}$ with additional structure data, we ask a natural question: is there some notion of algebraic theory whose category of algebras in $\mathcal{V}$ is equivalent to $C$? In the case we work in the category of sets, Lawvere theories introduced in \cite{Law2} are known to be sufficient to capture all cases where the structure data is finitary. Furthermore, every Lawvere theory induces a monadic adjunction between the category of sets and the category of algebras. 

Lawvere theories have a couple more properties that make them pleasant to work with in the category of sets. Every Lawvere theory admits a presentation by generators and relations given by the counit of an adjunction between reduced signatures and Lawvere theories (see Definition \ref{def_of_present} and Lemma \ref{Law_th_are_presentable}). Often times, we may minimize the number of generators used in a presentation to minimize the data. Moreover, given a span of Lawvere theories
\[
L_1\xleftarrow{f}L_0\xrightarrow{g}L_2
\]
(or any connected diagram of Lawvere theories), we have a good grip on how to construct the colimit by utilizing presentations of generators and relations. 

When we move to a non-set based setting, our story becomes more complicated. Lawvere theories are no longer sufficient, even up to homotopy. While it is not explicitly written anywhere from our knowledge, it is well-known that there is no Lawvere theory whose homotopy algebras of \cite{bern1} in spaces (where spaces means simplicial sets or topological spaces) model the $E_\infty$-spaces of \cite{May2}. This problem extends to trying to capture symmetric monoidal categories as homotopy algebras over a Lawvere theory in the category of small categories. The level of difficulty lies in the way that the symmetry isomorphisms interacts with the diagonal maps. In a general symmetric monoidal category $(C,\otimes, I,\alpha,\lambda,\rho,\beta)$, 
\[
\beta_{x,x}\neq 1_{x\otimes x}.
\]
In the setting of homotopy algebras, this axiom is forced to hold up to a higher cell.

Enriched Lawvere theories of \cite{Power1999} are used to circumvent this issue. Classically, the first write up for this idea originates from \cite{Gray1974} and was applied to homotopical methods in \cite{SchwanzlVogt1989}. We focud in on the work of Schwanz and Vogt.  Schwanz and Vogt used this framework to prove that when an $E_\infty$-space is equipped with a notion of homotopy coherent inverse, as introduced by their paper, then it is equivalent to a topological abelian group. Therefore the notion of homotopy coherent inverses, as discussed by Schwanz and Vogt is insufficient in capturing $E_\infty$-spaces. At first glance, we might assume that no notion of homotopy coherent inverses could exist.  Upon further examination, Schwanz and Vogt make the assumption that the preimage of the $k$-fold multiplication operations and the inverse operation are contractible. As they point out, this condition is too strict.

We are inspired from the framework of Lawvere $2$-theories (see \cite{Perutka2026} and \cite{yan1}). In this setting, there is a span of Lawvere $2$-theories
\[
L_{gp}\xleftarrow{i}L_{mon}\xrightarrow{j}L_{ab},
\]
whose algebras in the category of groupoids are coherent $2$-groups, monoidal groupoids, and symmetric monoidal groupoids, respectively. The maps $i$ and $j$ induce the forgetful functors upon taking algebras. Moreover, the pushout produces a Lawvere $2$-theory $L_{pic}$ whose algebras in the category of groupoids are coherent Picard groupoids. The $1$-truncation of the sphere spectrum, i.e., the free skeletal and permutative Picard groupoid on one object, is known to have hom-group on the generating object given by $\mathbf{Z}/2$. On the level of Lawvere $2$-theories, this says there is a non-trivial $2$-cell in $L_{pic}$ between $\text{id}_1$ and itself. Therefore, the contractibility condition of \cite{SchwanzlVogt1989} will not be met when we translate to simplicial sets or topological spaces via the nerve or classifying space functor.

We begin in Section 2 by detailing common notation that we use throughout the paper. We move on to provide background on algebraic theories needed for this paper. We make use of signatures, pros, and Lawvere theories. To conclude, we provide background for the theory of model structures on categories. 

In Section 3, we provide background on results arising in enriched category theory. A standard reference for the subject matter is Kelly in \cite{Kelly3}. We take much of that section straight from \cite{Kelly3} and provide it for completion purposes. We recommend anybody familiar with enriched category theory to skip Section 3.
For those unfamiliar or rusty, Section 3 provides an outline on what enriched category theory is and fundamental tools of enriched category theory used throughout the literature such as change-of-basis functors, copowers, the enriched tensor product, and the enriched hom functor. 

In Section~4, we first introduce controlled theories and their basic properties. Controlled theories and connected diagrams of controlled theories were originally introduced by us in \cite{Taylor2026}. There, we defined a functor assigning to each connected diagram of controlled theories a limit sketch that syntactically weakens the axioms of the associated Lawvere theory according to the specified control pro. We then define the notion of an admissible controlled theory, a distinguished class of controlled theories whose data is determined entirely by a pro that faithfully embeds into the initial Lawvere theory, a reduced signature, and a congruence relation on the free Lawvere theory generated by that signature.  We show that every admissible controlled theory functorially determines a controlled theory and that this construction defines a right adjoint.

 In Section 5, we discuss facets of enriched categorical algebra. We begin by defining enriched signatures and enriched pros, which was used by Bressie in \cite{Bressie2020} in his work to provide categorical algebraic variants of globular operads of \cite{Lein}. We go on to provide background on enriched Lawvere theories. Enriched pros and Lawvere theories serve as the backbone of what follows. We define a notion of deformation of a pro in a locally presentable cartesian closed category and show how a deformation functorially induces an enriched pro. We finish the section by defining a notion of deformation of a controlled theory and show how a deformation functorially induces an enriched Lawvere theory.
 We were inspired to define the notion of deformation of a pro and controlled theory by the augmentations that May used to define $E_\infty$-spaces in \cite{May2} and the work of Schwanz and Vogt in \cite{SchwanzlVogt1989}.

In the sixth section, we construct and study the one-dimensional
categorification and homotopification of controlled theories. We begin by
introducing lax, pseudo, and strict augmentations of a controlled theory in a
cartesian closed model category. We then define the notion of an
$\Omega^*$-free deformation of a controlled theory $\Omega$, which serves as an analogue of $\Sigma$-free operads in the setting of
deformations over a controlled theory.

We next construct a functor
\[
\mathcal{AR}_{(-)}:\mathbf{cTh}\longrightarrow 2\mathbf{Law},
\]
called the \emph{algebraic realization functor}. This functor assigns to a
controlled theory $\Omega$ a Lawvere $2$-theory obtained by syntactically
weakening the axioms of $\Omega$ at the locations prescribed by its control
pro. The main result of this construction is the following theorem.

Theorem \ref{proof_of_freeness_of_alg_aug}.  Let
\[
\Omega:=\langle \mathcal{G}:P\xrightarrow{i}\operatorname{Fr}(\mathcal{G})
\xrightarrow{\mathbf{st}}L\rangle
\]
be a controlled theory. The algebraic augmentation of $\Omega$ is an
$\Omega^*$-free strong augmentation of $\Omega$.

We extend the domain of the algebraic realization functor to the category of
connected diagrams of controlled theories.

Finally, we modify the construction of the algebraic realization functor by
applying the nerve functor to obtain the \emph{nerve realization functor}. We
show that there exist controlled theories whose associated categories of
algebras recover both $A_\infty$-spaces and $E_\infty$-spaces. As an
application, we construct a new model category $\mathcal{A}^{gl}_{\infty}\text{-Spaces}$ and prove the following theorem.

\emph{Theorem \ref{theorem_sixpointtwenty}. The category $\mathcal{A}^{gl}_{\infty}\text{-Spaces}$,
equipped with the projective model structure, is a model for $\infty$-groups.}

Theorem \ref{theorem_sixpointtwenty} establishes that
$\mathcal{A}^{gl}_{\infty}\text{-Spaces}$ provides a model for
$\infty$-groups. Moreover, we construct a new model category $\mathcal{E}^{gl}_{\infty}\text{-Spaces}$
defined by the pullback of the diagram
\[
\mathcal{A}^{gl}_{\infty}\text{-Spaces}
\xrightarrow{U}
A_{\infty}\text{-Spaces}
\xleftarrow{U}
E_{\infty}\text{-Spaces}.
\]
 We conclude by proving the following proposition, which demonstrates
that the strictification morphism considered below fails to be a weak
equivalence.

\emph{Theorem \ref{not_we_eq}. The morphism of theories
\[
\mathbf{st}:\mathcal{NR}_{\Omega_{pic}}\longrightarrow D(\mathcal{AB})
\]
is not a weak equivalence in $\mathbf{sSTh}$.}

The motivation for this paper is twofold. First, we aim to introduce
controlled theories to the broader mathematical community. The applications of
controlled theories within our research program continue to expand, and their
structural role is essential for constructing the appropriate higher
categorical models in our research. Second, we construct a theory that is not weakly equivalent
to the theory of simplicial abelian groups, while whose algebras form a
subcategory of $E_\infty$-spaces equipped with coherent
invertibility data.

\subsection*{Relationship to other work}
In the case of operads, the thesis work of Gould in \cite{Gould2008} provides a functorial categorification of the data of an operad. Our work here generalizes to more a general class of algebraic theories. In particular, we provide a methodology that allows us to obtain coherent $2$-groups as a categorification of a particular theory.

In the future, we will show how to obtain a modification of the work of Gould with our machinery. An operad will functorially produce a controlled theory. When we take the symmetric operad for commutative monoids and apply the algebraic realization to the controlled theory produced Gould's construction produces biased, we obtain a Lawvere $2$-theory whose algebras are unbiased symmetric monoidal categories. Gould's construction produces an operad whose algebras are biased symmetric monoidal categories. 

\subsection*{Special thanks}
We would like to thank Nathanael Arkor. Nathanael Arkor reached out and gave some editing advice and gave some references to results that we were not familiar with. Another special thanks goes to Tony Elmendorf who gave us the idea for this paper to begin with. While we were at dinner with Tony, he said he wondered if homological group-completion can be made monadic in a more point-set sense. While we believe the answer to that question is no, the question did ultimately lead to the formation of this paper. 

\subsection{Remarks about edited version}
This document is an edited version of work previously submitted on ArXiv. We have reworked it to be more formal and less ambitious about our future results. We still will put out those promised results; however, those types of statements are too informal for a published piece of math literature. 
 
\section{Notation and background}

\subsection*{Common notation}

We fix the following notation for the remainder of this paper.

\begin{itemize}
    \item $*$ denotes the terminal object of a category, whenever it exists.
    \item $\mathbf{N}$ is the natural numbers, including $0$.
    \item $\Delta$ denotes the category of finite ordinals.
    \item $\mathbf{Set}$ denotes the category of sets.
    \item $\mathbf{Mon}$ denotes the category of monoids.
    \item $\mathbf{cMon}$ denotes the category of commutative monoids.
    \item $\mathbf{Grp}$ denotes the category of groups.
    \item $\mathbf{sSet}$ denotes the category of simplicial sets.
    \item $\mathbf{sMon}$ denotes the category of simplicial monoids.
    \item $\mathbf{sGrp}$ denotes the category of simplicial groups.
    \item $\mathbf{Cat}$ denotes the category of small categories and functors.
    \item $\mathbf{ob}:\mathbf{Cat}\to\mathbf{Set}$ denotes the functor sending a category to its set of objects.
    \item $E:\mathbf{Set}\to\mathbf{Cat}$ denotes the right adjoint to the object functor $\mathbf{ob}$.
    \item $B:\mathbf{Grp}\to\mathbf{Cat}$ denotes the functor sending a group to the corresponding one-object groupoid.
    \item $\mathbf{Top}$ denotes the category of compactly generated weakly Hausdorff spaces.
    \item $|-|:\mathbf{sSet}\to\mathbf{Top}$ denotes the geometric realization functor.
    \item $S:\mathbf{Top}\to\mathbf{sSet}$ denotes the singular simplices functor.
    \item $N:\mathbf{Cat}\to\mathbf{sSet}$ denotes the nerve functor.
    \item $\mathbf{CAT}$ denotes a large category of categories bounded by a sufficiently large regular cardinal
\end{itemize}


\subsection*{Algebraic theories}

We now outline some background on set-based algebraic theories used throughout this paper. We begin by defining signatures and reduced signatures.

\begin{definition}
A \emph{signature} is a functor
\[
S:\mathbf{N}^{2}\to \mathbf{Set}.
\]
\end{definition}

The category of signatures, denoted by $\mathbf{Sig}$, is the functor category
\[
[\mathbf{N}^{2},\mathbf{Set}].
\]
Our notion of signature agrees with the notion of a prop signature given in Definition 5.25 of \cite{FongSpivak2019}. To see this, let $G$ be a set equipped with functions
\[
\begin{tikzcd}
G
  \arrow[r, shift left=0.6ex, "s"]
  \arrow[r, shift right=0.6ex, swap, "t"]
&
\mathbf{N}
\end{tikzcd}
\]

Define a functor $G':\mathbf{N}^{2}\to\mathbf{Set}$ by setting
\[
G'_{n,m}=\{g\in G:s(g)=n,\ t(g)=m\}
\]
for all $n,m\geq 0$. Conversely, given a signature
\[
c:\mathbf{N}^{2}\to\mathbf{Set},
\]
we may define a set
\[
c'=\coprod_{n,m\geq 0}c_{n,m}
\]
together with functions
\[
\begin{tikzcd}
c'
  \arrow[r, shift left=0.6ex, "s"]
  \arrow[r, shift right=0.6ex, swap, "t"]
&
\mathbf{N}
\end{tikzcd}
\]
induced by the indexing of the coproduct.

A signature $S:\mathbf{N}^{2}\to\mathbf{Set}$ is \emph{reduced} if
\[
S(n,k)=\emptyset
\]
whenever $k\neq 1$. We denote by $\mathbf{rSig}$ the full subcategory of $\mathbf{Sig}$ consisting of reduced signatures. We immediately observe that
\[
\mathbf{rSig}\simeq [\mathbf{N},\mathbf{Set}].
\]

\begin{lemma}
The categories $\mathbf{Sig}$ and $\mathbf{rSig}$ are locally presentable.
\end{lemma}

\begin{proof}
Both being categories of models over a limit sketch, this follows immediately from Corollary 1.52 of \cite{AdRo}.
\end{proof}

In order to define controlled theories, we require the use of pros.

\begin{definition}
A \emph{pro} is a strict monoidal category $P$ whose underlying monoid of objects is
\[
(\mathbf{N},+,0).
\]
\end{definition}

Pros encode the data necessary to describe operations of the form
\[
n\to k
\]
and serve as a generalization of ordinary operads.

\begin{example}\label{examples of pros}
The following are examples of pros.
\begin{itemize}
    \item The discrete category on $\mathbf{N}$ is the initial pro.
    \item The poset $(\mathbf{N},\leq)$ is a pro.
    \item If $P$ is a pro, then $P^{\mathrm{op}}$ is a pro.
    \item The category
    \[
    \Sigma=\coprod_{n=0}^{\infty}B\Sigma_n
    \]
    is a pro.
    \item The opposite of the skeleton of the category of finite sets is a pro. We denote this pro by $\mathcal{F}$.
    \item Let $E$ be a category with a chosen collection of finite products, and let $x\in\mathbf{ob}(E)$. There is a pro $\mathbf{End}(x)$ whose objects are the finite powers
    \[
    x^n=\prod_{i=1}^{n}x,
    \]
    which are naturally indexed by $\mathbf{N}$. The hom-sets are given by
    \[
    \mathbf{End}(x)(m,n)=E(x^m,x^n).
    \]
    Composition is induced by composition in $E$. The monoidal structure is induced by the product structure of $E$, which corresponds to addition of natural numbers via the canonical identifications
    \[
    x^{n+m}\cong x^n\times x^m.
    \]
\end{itemize}
\end{example}

\begin{definition}
A map of pros
\[
f:P\to Q
\]
is a strict monoidal functor satisfying
\[
f(1)=1.
\]
\end{definition}

Most authors define a map of pros to be a strict monoidal functor that is the identity on objects. For pros, these definitions are equivalent since the objects are determined by the monoidal unit and tensor product.

\begin{notation}
We write $\mathbf{Pro}$ for the category of pros.
\end{notation}

Every pro $P$ has an underlying reduced signature, denoted by $UP$, given by
\[
UP(n)=P(n,1).
\]
This defines the object assignment of a functor
\[
U:\mathbf{Pro}\to\mathbf{rSig}.
\]
The functor $U$ is right adjoint to the free pro construction
\[
\mathbf{N}[-]:\mathbf{rSig}\to\mathbf{Pro}.
\]
An explicit construction of the left adjoint and the corresponding adjunction is given in our thesis \cite{Taylor2026}.

\begin{lemma}
The category $\mathbf{Pro}$ is locally presentable.
\end{lemma}

\begin{proof}
The functor
\[
U:\mathbf{Pro}\to\mathbf{rSig}
\]
is monadic, and $\mathbf{Pro}$ is therefore the category of algebras for a finitary monad on the locally presentable category $\mathbf{rSig}$. Hence $\mathbf{Pro}$ is locally presentable.
\end{proof}

We now need the following construction. Let $\mathcal{G}$ be a reduced signature and $P$ be a pro. We define $P[G]$ to be the pushout of the following span of pros.
\[
P\xleftarrow{!}\mathbf{N}\xrightarrow{!}\mathbf{N}[\mathcal{G}]
\]

\begin{lemma}
There is a functor 
\[
(-)[-]:\mathbf{Pro}\times\mathbf{rSig}\to\mathbf{Pro}
\]
defined on objects by sending a pair $(P,\mathcal{G})$ to $P[\mathcal{G}]$.
\end{lemma}

We now construct a pro that will be used in the definition of a controlled theory whose algebras will be coherent $2$-groups.
\begin{itemize}
    \item Freely adjoin a morphism
    \[
    \Delta_k:1\to k
    \]
    to $\mathbf{N}$ for every $k\geq0$, obtaining the pro $\mathbf{N}^{\Delta}$.
    
    \item Define a map of pros
    \[
    j:\mathbf{N}^{\Delta}\to\mathcal{F}
    \]
    by sending $\Delta_k$ to the $k$-fold diagonal map.
    
    \item Define a cocongruence relation $\mathcal{L}$ on $\mathbf{N}^{\Delta}$ by declaring
    \[
    (f,g)\in\mathcal{L}
    \]
    if and only if
    \[
    j(f)=j(g).
    \]
    
    \item Define $\nabla$ to be the quotient pro
    \[
    \nabla=\mathbf{N}^{\Delta}/\mathcal{L},
    \]
    and let
    \[
    q:\mathbf{N}^{\Delta}\to\nabla
    \]
    denote the quotient map.
\end{itemize}

\begin{lemma}\label{constr_of_nabla}
The category $\nabla$ is a pro.  Moreover, it comes equipped with a faithful map of pros.
\[
i:\nabla\to\mathcal{F}.
\]
\end{lemma}

\begin{proof}
By construction, $\nabla$ is a pro. By definition, $j(f)=j(g)$ for all $(f,g)\in\mathcal{L}$. By the universal property of the quotient, there is a unique map of pros 
\[
i:\nabla\to\mathcal{F}
\]
such that
\[
i\circ q=j.
\]
\end{proof}

We interpret the pro $\nabla$ as the pro encoding all diagonal operations. We conclude this subsection by recalling some preliminary facts about Lawvere theories.

The category of Lawvere theories, denoted by $\mathbf{Law}$, is the category of theories out of $\mathcal{F}$. There is a forgetful adjunction
\[
U:\mathbf{Law}\to\mathbf{rSig}
\]
with left adjoint
\[
\mathbf{Fr}:\mathbf{rSig}\to\mathbf{Law}.
\]
This adjunction gives the following consequence.

\begin{lemma}
The category of Lawvere theories $\mathbf{Law}$ is locally presentable.
\end{lemma}

\begin{proof}
The category $\mathbf{Law}$ is the category of algebras for the finitary monad induced by the free-forgetful adjunction. Since $\mathbf{rSig}$ is locally presentable, it follows that $\mathbf{Law}$ is locally presentable.
\end{proof}

\begin{definition}\label{def_of_present}
A \emph{presentation} of a Lawvere theory $L$ consists of a reduced signature $S$ together with a full map of Lawvere theories
\[
\phi:\mathbf{Fr}(S)\to L.
\]
A Lawvere theory is called \emph{presentable} if it admits such a presentation.
\end{definition}

\begin{lemma}\label{Law_th_are_presentable}
Every Lawvere theory is presentable.
\end{lemma}

\begin{proof}
Let $L$ be a Lawvere theory. The counit of the free-forgetful adjunction,
\[
\epsilon_L:\mathbf{Fr}(U(L))\to L,
\]
provides a presentation of $L$.
\end{proof}

There is also a left adjoint
\[
L_{(-)}:\mathbf{Pro}\to\mathbf{Law}
\]
that factors the left adjoint
\[
\mathbf{Fr}:\mathbf{rSig}\to\mathbf{Law}
\]
through
\[
\mathbf{N}[-]:\mathbf{rSig}\to\mathbf{Pro}.
\]

\begin{definition}
Let $L=(L,J)$ be a Lawvere theory. An \emph{algebra} over $L$ consists of a set $X$ together with a map of Lawvere theories
\[
A:L\to\mathbf{End}(X).
\]
\end{definition}

We denote the category of algebras over $L$ by $\mathbf{Alg}(L)$. Traditionally, Lawvere theories are described in terms of their models rather than their algebras. The categories of models and algebras associated to a Lawvere theory are equivalent, but the algebraic description is more convenient for our purposes. The following result describes the behavior of algebras under connected colimits of Lawvere theories.

\begin{theorem}\label{Alg_pre_conn_colimits}
Let
\[
J:I\to\mathbf{Law}
\]
be a connected diagram of Lawvere theories. Then there is an isomorphism of categories
\[
\mathbf{Alg}\left(\operatorname{colim}_{i\in\mathbf{ob}(I)}J(i)\right)
\cong
\lim
\left(
I^{\mathrm{op}}
\xrightarrow{J^{\mathrm{op}}}
\mathbf{Law}^{\mathrm{op}}
\xrightarrow{\mathbf{Alg}(-)}
\mathbf{CAT}
\right).
\]
\end{theorem}

\begin{proof}
This follows directly from Example 2.10 of \cite{ArkorMcDermott2025}.
\end{proof}

We have chosen not to work primarily with (colored) operads (see \cite{Yau2016}) or (colored) PROPs. The frameworks of pros and Lawvere theories provide the level of generality needed for this paper while retaining the structure necessary for the constructions involving controlled theories. Moreover, the presence of symmetric group actions inherent in operadic and prop-based frameworks introduces additional complications that are unnecessary for our purposes. These symmetries can obstruct the controlled structures we wish to encode, as discussed in the introduction.

\subsection*{Model structures}
We recall some background on model structures. The original definition of a model
structure was introduced by Quillen in \cite{Quillen67}, where only the
existence of finite limits and colimits was required. The modern formulation of
a model category was developed in \cite{DwyerHirschhornKan1997}.

\begin{definition}
A \emph{model category} is a bicomplete category $\mathcal{M}$ equipped with
three distinguished classes of morphisms
\[
\mathcal{W},\qquad C,\qquad \mathcal{F}
\]
called weak equivalences, cofibrations, and fibrations, respectively, such that
\[
(C,\mathcal{F}\cap\mathcal{W})
\qquad\text{and}\qquad
(C\cap\mathcal{W},\mathcal{F})
\]
are weak factorization systems on $\mathcal{M}$.
\end{definition}

A morphism that is both a weak equivalence and a cofibration is called a
\emph{trivial cofibration}. Dually, a morphism which is both a weak equivalence
and a fibration is called a \emph{trivial fibration}.

An object $X$ of $\mathcal{M}$ is \emph{fibrant} if the unique morphism
\[
X\longrightarrow *
\]
to the terminal object is a fibration. Dually, $X$ is \emph{cofibrant} if the
unique morphism
\[
\emptyset\longrightarrow X
\]
from the initial object is a cofibration.

Model categories provide a $1$-categorical presentation of
$(\infty,1)$-categories in the sense of \cite{RiehlVerity2020}. Although weaker
notions of homotopical structure exist, such as semi-model categories
\cite{hov3}, model categories provide a sufficiently general framework for our
purposes.

\begin{example}
The category of simplicial sets $\mathbf{sSet}$ admits several model
structures, including the Quillen model structure \cite{JarGo} and the Joyal
model structure \cite{Joyal1}.
\end{example}

\begin{example}
The category of topological spaces $\mathbf{Top}$ admits several model
structures, including the Quillen model structure \cite{Quillen67} and the
Strøm model structure \cite{Strom}.
\end{example}

\begin{example}
The category of groupoids $\mathbf{Gpd}$ admits the canonical model structure
\cite{And1}.
\end{example}

\begin{definition}
Let $\mathcal{M}$ and $\mathcal{N}$ be model categories. A
\emph{Quillen adjunction} is an adjunction
\[
\begin{tikzcd}[column sep=large]
\mathcal{M}
\arrow[r, shift left=1.2ex, "L"]
&
\mathcal{N}
\arrow[l, shift left=1.2ex, "R"]
\end{tikzcd}
\]
such that the left adjoint $L$ preserves cofibrations and trivial
cofibrations.
\end{definition}

Quillen adjunctions can be characterized as adjunctions that satisfy various equivalent set of conditions. We specify this one as to provide a definition.

\begin{example}
There is a Quillen adjunction
\[
\begin{tikzcd}[column sep=large]
\mathbf{Gpd}
\arrow[r, shift left=1.2ex, "B"]
&
\mathbf{Top}
\arrow[l, shift left=1.2ex, "\Pi_1"]
\end{tikzcd}
\]
between groupoids and topological spaces, where $B$ is the classifying space
functor and $\Pi_1$ is the fundamental groupoid functor.
\end{example}

\begin{definition}
Let
\[
\begin{tikzcd}[column sep=large]
\mathcal{M}
\arrow[r, shift left=1.2ex, "L"]
&
\mathcal{N}
\arrow[l, shift left=1.2ex, "R"]
\end{tikzcd}
\]
be a Quillen adjunction. It is a \emph{Quillen equivalence} if, for every
cofibrant object $X$ of $\mathcal{M}$ and every fibrant object $Y$ of
$\mathcal{N}$, a morphism
\[
f:LX\longrightarrow Y
\]
is a weak equivalence in $\mathcal{N}$ if and only if its adjoint
\[
X\xrightarrow{\eta_X}RLX\xrightarrow{R(f)}RY
\]
is a weak equivalence in $\mathcal{M}$.
\end{definition}

 \section{Basics of enriched categorical structures}

 \subsection*{Enriched categories}
Fix a symmetric monoidal category $(\mathcal{V},\otimes,I)$.

\begin{definition}
A \emph{$\mathcal{V}$-enriched category} $C$ consists of the following data:
\begin{itemize}
    \item a collection of objects $C_0$, whose elements are denoted by
    $x,y,z,\dots$;
    \item for every pair of objects $x,y\in C_0$, an object
    $C(x,y)$ of $\mathcal{V}$;
    \item for every triple of objects $x,y,z\in C_0$, a composition morphism
    \[
    \circ:C(y,z)\otimes C(x,y)\longrightarrow C(x,z);
    \]
    \item for every object $x\in C_0$, a unit morphism
    \[
    j_x:I\longrightarrow C(x,x).
    \]
\end{itemize}

This data is required to satisfy the following axioms.
\begin{itemize}
    \item \textbf{Associativity.}
    For all objects $w,x,y,z\in C_0$, the following diagram commutes,
    where $\alpha$ denotes the associativity isomorphism in $\mathcal V$.
    
    \[
    \begin{tikzpicture}
        \node (A) at (0,0)
        {$\bigl(C(z,w)\otimes C(y,z)\bigr)\otimes C(x,y)$};

        \node (B) at (6,0)
        {$C(z,w)\otimes\bigl(C(y,z)\otimes C(x,y)\bigr)$};

        \node (C) at (0,-2)
        {$C(z,w)\otimes C(x,z)$};

        \node (D) at (6,-2)
        {$C(z,w)\otimes C(x,z)$};

        \node (E) at (3,-3.5)
        {$C(x,w)$};

        \draw[->] (A) -- node[above] {$\alpha$} (B);
        \draw[->] (A) -- node[left] {$\circ\otimes 1$} (C);
        \draw[->] (B) -- node[right] {$1\otimes\circ$} (D);
        \draw[->] (C) -- node[below left] {$\circ$} (E);
        \draw[->] (D) -- node[below right] {$\circ$} (E);
    \end{tikzpicture}
    \]

    \item \textbf{Unit.}
    For all objects $x,y\in C_0$, the following diagrams commute, where
    $\lambda$ and $\rho$ denote the left and right unitors in $\mathcal V$.
    
    \[
    \begin{tikzpicture}
        \node (A) at (0,0) {$I\otimes C(x,y)$};
        \node (B) at (4,0) {$C(x,x)\otimes C(x,y)$};
        \node (C) at (2,-2) {$C(x,y)$};

        \draw[->] (A)--node[above]{$j_x\otimes1$}(B);
        \draw[->] (A)--node[left]{$\lambda$}(C);
        \draw[->] (B)--node[right]{$\circ$}(C);
    \end{tikzpicture}
    \]

    \[
    \begin{tikzpicture}
        \node (A) at (0,0) {$C(x,y)\otimes I$};
        \node (B) at (4,0) {$C(x,y)\otimes C(y,y)$};
        \node (C) at (2,-2) {$C(x,y)$};

        \draw[->] (A)--node[above]{$1\otimes j_y$}(B);
        \draw[->] (A)--node[left]{$\rho$}(C);
        \draw[->] (B)--node[right]{$\circ$}(C);
    \end{tikzpicture}
    \]
\end{itemize}
\end{definition}

\begin{example}
Let $\mathbf B$ be the Boolean category, obtained from the walking morphism
by equipping it with the operation $\vee$. Its composition is determined by
the table
\[
\begin{array}{c|cc}
\vee & 0 & 1\\
\hline
0&0&1\\
1&1&1
\end{array}
\]
A $\mathbf B$-enriched category is precisely a poset.
\end{example}

\begin{example}
Let
\[
[0,\infty)=([0,\infty),\leq,0,+)
\]
be the commutative monoid poset regarded as a strict permutative category.
A $[0,\infty)$-enriched category is a Lawvere metric space
(see \cite{Lawvere1973}).
\end{example}

\begin{example}
Let $(\mathcal{V},\otimes,I)$ be a symmetric monoidal closed category. There is
a canonical $\mathcal{V}$-enriched category, denoted by
$\underline{\mathcal{V}}$, whose objects are the objects of $\mathcal{V}$ and
whose hom-objects are given by
\[
\underline{\mathcal{V}}(v,v')=[v,v']
\]
for all $v,v'\in\mathrm{ob}(\mathcal{V})$.

The unit morphism
\[
j_v:I\to \underline{\mathcal{V}}(v,v)
\]
is the image of the left unitor
\[
\lambda_v:I\otimes v\to v
\]
under the tensor-hom adjunction.

The composition morphism
\[
\circ:
\underline{\mathcal{V}}(y,z)\otimes
\underline{\mathcal{V}}(x,y)
\longrightarrow
\underline{\mathcal{V}}(x,z)
\]
is induced by the evaluation morphism
\[
\mathrm{ev}:x\otimes[x,y]\to y,
\]
which is the image of
\[
\mathrm{id}_{[x,y]}:[x,y]\to[x,y]
\]
under the tensor-hom adjunction
\[
\mathcal{V}([x,y],[x,y])
\cong
\mathcal{V}(x\otimes[x,y],y).
\]

Explicitly, the composition morphism is the transpose of the composite
\[
[y,z]\otimes[x,y]\otimes x
\longrightarrow
[y,z]\otimes y
\longrightarrow
z,
\]
where the first map is induced by evaluation and the second map is the
evaluation morphism for $[y,z]$.
\end{example}

\begin{example}
A $\mathbf{Cat}$-enriched category is precisely a $2$-category.
\end{example}

\begin{notation}
Let $(\mathcal{V},\otimes,I)$ be a symmetric monoidal category, let $C$ be a
$\mathcal{V}$-category, and let $x,y\in C_0$. Given a morphism
\[
f:I\to C(x,y),
\]
we define the \emph{pre-composition morphism}
\[
f^*:C(y,z)\to C(x,z)
\]
for every $z\in C_0$ to be the composite
\[
C(y,z)
\cong
C(y,z)\otimes I
\xrightarrow{1\otimes f}
C(y,z)\otimes C(x,y)
\xrightarrow{\circ}
C(x,z).
\]

Similarly, for every $w\in C_0$, the \emph{post-composition morphism}
\[
f_*:C(w,x)\to C(w,y)
\]
is defined by the composite
\[
C(w,x)
\cong
I\otimes C(w,x)
\xrightarrow{f\otimes1}
C(x,y)\otimes C(w,x)
\xrightarrow{\circ}
C(w,y).
\]
\end{notation}

\begin{definition}
Let $(\mathcal{V},\otimes,I)$ be a symmetric monoidal category, and let $C$
and $D$ be $\mathcal{V}$-enriched categories. A
\emph{$\mathcal{V}$-enriched functor}
\[
F:C\to D
\]
consists of

\begin{itemize}
    \item an object $Fx\in D_0$ for every object $x\in C_0$, and
    \item a morphism in $\mathcal{V}$
    \[
    F_{x,y}:C(x,y)\to D(Fx,Fy)
    \]
    for every pair of objects $x,y\in C_0$.
\end{itemize}

This data satisfies the following axioms.
\begin{itemize}
    \item \textbf{Preservation of composition.}
    For all objects $x,y,z\in C_0$, the following diagram commutes.
    
    \[
    \begin{tikzpicture}[node distance=4cm]
        \node (A) at (0,0)
        {$C(y,z)\otimes C(x,y)$};

        \node (B) at (6,0)
        {$D(Fy,Fz)\otimes D(Fx,Fy)$};

        \node (C) at (0,-1.5)
        {$C(x,z)$};

        \node (D) at (6,-1.5)
        {$D(Fx,Fz)$};

        \draw[->] (A)--node[above]{$F_{y,z}\otimes F_{x,y}$}(B);
        \draw[->] (A)--node[left]{$\circ$}(C);
        \draw[->] (B)--node[right]{$\circ$}(D);
        \draw[->] (C)--node[below]{$F_{x,z}$}(D);
    \end{tikzpicture}
    \]

    \item \textbf{Preservation of units.}
    For every object $x\in C_0$, the following diagram commutes.
    
    \[
    \begin{tikzpicture}
        \node (A) at (0,0) {$I$};
        \node (B) at (4,0) {$C(x,x)$};
        \node (C) at (2,-2) {$D(Fx,Fx)$};

        \draw[->](A)--node[above]{$j_x^C$}(B);
        \draw[->](A)--node[left]{$j_{Fx}^D$}(C);
        \draw[->](B)--node[right]{$F_{x,x}$}(C);
    \end{tikzpicture}
    \]
\end{itemize}
\end{definition}

\begin{example}
A $\mathbf{B}$-enriched functor is precisely an order-preserving map between
posets.
\end{example}

\begin{definition}
Let $(\mathcal{V},\otimes,I)$ be a symmetric monoidal category, let $C$ and
$D$ be $\mathcal{V}$-enriched categories, and let
\[
F,G:C\to D
\]
be $\mathcal{V}$-enriched functors. A
\emph{$\mathcal{V}$-natural transformation}
\[
\alpha:F\Rightarrow G
\]
consists of a morphism
\[
\alpha_x:I\to D(Fx,Gx)
\]
for every object $x\in C_0$ such that, for every pair of objects
$x,y\in C_0$, the following diagram commutes, where $\tilde{G}_{x,y}=\lambda^{-1}\circ G_{x,y}$.
\[
\begin{tikzpicture}
    \node (A) at (0,0)
    {$C(x,y)$};

    \node (B) at (4,0)
    {$D(Fx,Fy)$};

    \node (C) at (0,-2)
    {$I\otimes D(Gx,Gy)$};

    \node (D) at (4,-2)
    {$D(Fx,Gy)$};

    \draw[->] (A)--node[above]{$F_{x,y}$}(B);
    \draw[->] (A)--node[left]{$\tilde{G}_{x,y}$}(C);
    \draw[->] (B)--node[right]{$(\alpha_y)_*$}(D);
    \draw[->] (C)--node[below]{$\alpha_x^*$}(D);
\end{tikzpicture}
\]
Here $\alpha_x^*$ and $(\alpha_y)_*$ denote the pre- and post-composition
morphisms defined above.
\end{definition}

\begin{example}
A $\mathbf{B}$-natural transformation between order-preserving maps
\[
f,g:P\to Q
\]
is precisely an inequality
\[
f\leq g
\]
in the pointwise ordering. Thus, the set of order-preserving maps from $P$ to
$Q$ inherits the structure of a poset by declaring
\[
f\leq g
\quad\Longleftrightarrow\quad
f(x)\leq g(x)
\]
for every $x\in P$.
\end{example}

The following lemmas are well-known. The first theorem is obtained by a diagram chase and writing out the definitions. The second lemma is Theorem 4.5 of \cite{KellyLack2001} and is more difficult to prove than the first lemma.

\begin{lemma}
Let $(\mathcal{V},\otimes,I)$ be a symmetric monoidal category. There is a
strict $2$-category
\[
\mathcal{V}\mathbf{Cat}
\]
whose objects are $\mathcal{V}$-enriched categories, whose $1$-morphisms are
$\mathcal{V}$-enriched functors, and whose $2$-morphisms are
$\mathcal{V}$-natural transformations.
\end{lemma}

\begin{lemma}
Let $\mathcal{V}$ be a locally presentable closed symmetric monoidal category.
Then the category underlying $\mathcal{V}\mathbf{Cat}$ is locally presentable.
\end{lemma}

\subsection*{Transfer functors}

\begin{notation}
Let $\mathcal{V}$ and $\mathcal{W}$ be symmetric monoidal categories, and let
\[
F:\mathcal{V}\longrightarrow\mathcal{W}
\]
be a lax monoidal functor. We denote by
\[
F_*:\mathcal{V}\mathbf{Cat}\longrightarrow\mathcal{W}\mathbf{Cat}
\]
the induced $2$-functor. On objects, $F_*$ sends a
$\mathcal{V}$-enriched category $C$ to the $\mathcal{W}$-enriched category
$F_*(C)$ defined by
\[
F_*(C)_0=C_0
\]
and
\[
F_*(C)(x,y)=F(C(x,y))
\]
for all objects $x,y\in C_0$.
\end{notation}

Induced transfer $2$-functors provide a convenient framework for comparing
categories enriched over different symmetric monoidal bases. The following
lemma is immediate.

\begin{lemma}\label{ind_adjunct_enriched}
Let $\mathcal{V}$ be a symmetric monoidal category, and suppose
\[
F:\mathbf{Set}\longrightarrow\mathcal{V}
\]
is a lax monoidal functor with right adjoint
\[
U:\mathcal{V}\longrightarrow\mathbf{Set}.
\]
If $U$ is also lax monoidal, then there is an induced $2$-adjunction
\[
\begin{tikzcd}[column sep=large]
\mathbf{Cat}
\arrow[r, shift left=1.2ex, "F_*"]
&
\mathcal{V}\mathbf{Cat}
\arrow[l, shift left=0.8ex, "U_*"]
\end{tikzcd}
\]
between enriched categories.
\end{lemma}

In many examples, the underlying category functor for
$\mathcal{V}$-enriched categories is induced by a lax monoidal functor
\[
U:\mathcal{V}\longrightarrow\mathbf{Set}
\]
which admits a lax monoidal left adjoint. A fundamental example is given by
preadditive categories. The forgetful functor
\[
U:\mathbf{Ab}\longrightarrow\mathbf{Set}
\]
is lax monoidal and admits the free abelian group functor as a lax monoidal
left adjoint.

Applying $U$ to the hom-objects of a preadditive category produces its
underlying ordinary category. Conversely, the left adjoint sends an ordinary
category to the preadditive category with the same objects and hom-abelian
groups obtained by applying the free abelian group functor to each hom-set.

\subsection*{Underlying category structures}
The underlying category construction is a special case of the induced transfer functors introduced in the previous subsection. Indeed, the hom-functor
\[
\mathcal{V}(I,-):\mathcal{V}\longrightarrow\mathbf{Set}
\]
is lax monoidal, and the corresponding induced $2$-functor
\[
(-)_0:\mathcal{V}\mathbf{Cat}\longrightarrow\mathbf{Cat}
\]
recovers the ordinary category underlying a $\mathcal{V}$-enriched category.

\begin{definition}
Let $(\mathcal V,\otimes,I)$ be a symmetric monoidal category, and let $C$ be a $\mathcal V$-enriched category. The \emph{underlying category} of $C$, denoted $C_0$, has the same objects as $C$, and for objects $x,y\in C$,
\[
C_0(x,y)=\mathcal V(I,C(x,y)).
\]
The identity morphism on an object $x$ is the unit map
\[
j_x:I\to C(x,x),
\]
and the composition of morphisms
\[
f:I\to C(x,y),\qquad
g:I\to C(y,z),
\]
is given by the composite
\[
I\cong I\otimes I
\xrightarrow{g\otimes f}
C(y,z)\otimes C(x,y)
\xrightarrow{\circ}
C(x,z),
\]
where $\circ$ denotes the enriched composition.
\end{definition}

\begin{example}
The underlying category of a $\mathbf{B}$-enriched category is a poset.
\end{example}

\begin{example}
If $\mathcal{V}$ is a symmetric monoidal closed category, then the underlying category of the self-enriched category $\underline{\mathcal{V}}$ is canonically isomorphic to $\mathcal{V}$. Indeed,
\[
\underline{\mathcal V}_0(x,y)
=\mathcal V(I,[x,y])
\cong\mathcal V(I\otimes x,y)
\cong\mathcal V(x,y),
\]
naturally in $x$ and $y$.
\end{example}

The underlying category construction extends naturally to $\mathcal{V}$-functors and $\mathcal{V}$-natural transformations.

\begin{definition}
Let $(\mathcal V,\otimes,I)$ be a symmetric monoidal category, and let
\[
F:C\to D
\]
be a $\mathcal V$-functor. The \emph{underlying functor}
\[
F_0:C_0\to D_0
\]
is defined by
\[
F_0(x)=F(x)
\]
on objects and by

\[
F(f)=F_{x,y}\circ f
\]
for all $x,y\in \mathbf{ob}(C)$ and $f\in C_0(x,y)$ on morphisms.
\end{definition}

\begin{definition}
Let
\[
F,G:C\to D
\]
be $\mathcal V$-functors, and let
\[
\alpha:F\Rightarrow G
\]
be a $\mathcal V$-natural transformation. The \emph{underlying natural transformation}
\[
\alpha_0:F_0\Rightarrow G_0
\]
is defined by
\[
(\alpha_0)_x=\alpha_x:I\to D(Fx,Gx)
\]
for each object $x\in C$.
\end{definition}

\subsection*{Copowers}
\begin{definition}
Let $(\mathcal V,\otimes,I)$ be a symmetric monoidal category closed category, $v$ be an object of $\mathcal{V}$, $C$ be a $\mathcal{V}$-category, and $c\in C_0$. The \emph{copower}, if it exists, is an object $v\cdot c\in C_0$ together with a natural isomorphism 
\[
\phi^{v,c}:C(v\cdot c,-)\Rightarrow \mathcal{V}(v,[c,-]).
\]
\end{definition}

\begin{example}
Let $C$ be a category, $A$ be a set, and $c$ be an object of $C$. The copower, if it exists, is defined as the colimit of the constant functor $*_c:D(A)\to C$. That is to say, the copower is defined as the coproduct
\[
A\cdot c=\coprod_{a\in A}c.
\]
\end{example}

\begin{lemma}\label{copower_functor}
Let $C$ be a locally small category with choices of arbitrarily indexed coproducts. Then the copower operation induces a functor
\[
\cdot:\mathbf{Set}\times C\to C
\]
defined on objects by $A\cdot v:=\coprod_{a\in A}v$.
\end{lemma}
The proof of functorality is completely dependent on having choices of arbitrarily indexed coproducts. This functor restricts to a functor 
\[
D:\mathbf{Set}\to \mathcal{V}
\]
defined on objects by $D(A)=A\cdot I$.

\begin{proposition} \label{copies_functor}
Let $(\mathcal{V},\otimes, I)$ be a symmetric monoidal category with a choice of arbitrarily indexed coproducts. Suppose that the functor $D:\mathbf{Set}\to\mathcal{V}$ is lax-monoidal. Then $D_*:\mathbf{Cat}\to\mathcal{V}\mathbf{Cat}$ is left adjoint to the underlying category functor. 
\end{proposition}

\begin{proof}
  If $D:\mathbf{Set}\to\mathcal{V}$ is lax-monoidal, then the induced functor $D_*:\mathbf{Cat}\to\mathcal{V}\mathbf{Cat}$ exists. Now define $U:\mathcal{V}\to\mathbf{Set}$ by $U=\mathcal{V}(I,-)$. The functor $U$ is functorial and lax-monoidal. Moreover, this functor induces the underlying category functor and is right adjoint to $D$. By Lemma \ref{ind_adjunct_enriched}, $D_*$ is left adjoint to the underlying category functor.
\end{proof}

Let us now suppose that $\mathcal{V}$ is a symmetric monoidal closed category. We write $n_\mathcal{V}=\coprod_{i=1}^nI$ and $(-)^{n_{\mathcal{V}}}:\mathcal{V}\to\mathcal{V}$ to denote the functor defined on objects by $b^{n_\mathcal{V}}=[n_\mathcal{V},b]$.

\begin{lemma}\label{ind_fun_iso_to_prod}
Suppose that $\mathcal{V}$ is equipped with a choice for all finite products. Then there is a natural isomorphism 
\[
(-)^{n_\mathcal{V}}\cong (-)^n.
\]
\end{lemma}

\begin{proof}
For all objects $a$ and $b$, we have bijections
\[
\mathcal{V}(a,b^{n_\mathcal{V}})=\mathcal{V}(a,[n_\mathcal{V},b])\cong \mathcal{V}(a\otimes n_\mathcal{V},b)\cong \mathcal{V}(a\otimes (n\cdot I),b)
\]
\[
\mathcal{V}(a\otimes \coprod_{i=1}^nI,b)\cong \mathcal{V}(a\otimes I,b)^n\cong\mathcal{V}(a,b^n).
\]
The third isomorphism is the only tricky isomorphism. It uses the fact that the tensor product in a symmetric monoidal closed category preserves colimits. 
By Yoneda, $b^{n_{\mathcal{V}}}\cong b^n$. This forces there to be a natural isomorphism 
\[
(-)^{n_\mathcal{V}}\cong (-)^n.
\]
This concludes our proof.
\end{proof}

Baez and Williams proved this lemma in \cite{BaezWilliams2020} when restricted to cartesian monoidal closed categories.

\subsection*{Tensor product of enriched categories}
Given a symmetric monoidal category $\mathcal{V}$, we now define a tensor product on the underlying category of the $2$-category $\mathcal{V}\mathbf{Cat}$. Let $C$ and $D$ be $\mathcal{V}$-categories. Define $C\otimes D$ to have object set $\mathbf{ob}(C)\times\mathbf{ob}(D)$ and hom-object 
\[
(C\otimes D)((c,d),(c',d'))=C(c,c')\otimes D(d,d').
\]
Let $c,c',c''\in\mathbf{ob}(C)$ and $d,d',d''\in\mathbf{ob}(D)$. Define the composition morphism 
\[
\circ:(C\otimes D)((c',d'),(c'',d'')\otimes (C\otimes D)((c,d),(c',d'))\to (C\otimes D)((c,d),(c'',d''))
\]
as the following composite.
\[
\begin{tikzpicture}[node distance=1.5cm]
    \node (A) {$(C\otimes D)((c',d'),(c'',d'')\otimes (C\otimes D)((c,d),(c',d'))$};
    \node (B)[below of=A] {$(C(c',c'')\otimes D(d',d''))\otimes(C(c,c')\otimes D(d,d'))$};
    \node (C) [below of=B]{$(C(c',c'')\otimes C(c,c'))\otimes (D(d',d'')\otimes D(d,d'))$};
    \node (D)[below of=C]{$C(c,c'')\otimes D(d,d'')$};
    \node (E)[below of=D]{$(C\otimes D)((c,d),(c'',d'')$};
    \draw[->] (A) to node[left=3]{$=$}(B);
    \draw[->] (B) to node[left=3]{$\cong$}(C);
    \draw[->] (C) to node[left=3]{$\circ\otimes\circ$}(D);
    \draw[->] (D) to node[left=3]{$=$}(E);
\end{tikzpicture}
\]

For all objects $c$ of $C$ and $d$ of $D$, define $j_{(c,d)}:I\to (C\otimes D)((c,d),(c,d))$ as the following composite.
\[
I\xrightarrow{\cong}I\otimes I\xrightarrow{j_c\otimes j_d}C(c,c)\otimes D(d,d)
\]
This turns $(C\otimes D,\circ,j)$ into a $\mathcal{V}$-category. Moreover, this serves as the object part of a tensor product functor 
\[
\otimes:\mathcal{V}\mathbf{Cat}^2\to\mathcal{V}\mathbf{Cat}.
\]

The unit of the tensor product we just defined is the $\mathcal{V}$-category $I$ with object set $\{0\}$ and enriched-hom given by $I(0,0)=I$.

\subsection*{Enriched functor category}
Given a bicomplete symmetric monoidal closed category, we show how to provide the category of
$\mathcal{V}$-categories to inherit a closed structure. Let $C$ and $D$ be $\mathcal{V}$-categories. We define the
$\mathcal{V}$-category $[C,D]$ as follows.

The objects of $[C,D]$ are the $\mathcal{V}$-functors
\[
F:C\longrightarrow D.
\]
For $\mathcal{V}$-functors 
\[
\begin{tikzcd}
C
  \arrow[r, shift left=0.6ex, "F"]
  \arrow[r, shift right=0.6ex, swap, "G"]
&
D,
\end{tikzcd}
\]
define the hom-object by
the enriched end
\[
[C,D](F,G)=\int_{c\in C}D(Fc,Gc).
\]
The end is equipped with projections
\[
\pi_c:[C,D](F,G)\longrightarrow D(Fc,Gc)
\]
for every object $c\in C$.

Let
\[
F,G,H:C\longrightarrow D
\]
be $\mathcal{V}$-functors. For each object $c\in C$, consider the composite
\begin{equation}\label{map_ind_by_proj_of_c}
\begin{tikzcd}[column sep=large]
[C,D](G,H)\otimes[C,D](F,G)
\arrow[r, "\pi_c\otimes\pi_c"]
&
D(Gc,Hc)\otimes D(Fc,Gc)
\arrow[r, "\circ"]
&
D(Fc,Hc).
\end{tikzcd}
\end{equation}
These morphisms form a dinatural family in $c$. Hence, by the universal
property of the end, there is a unique morphism
\[
\circ:
[C,D](G,H)\otimes[C,D](F,G)
\longrightarrow
[C,D](F,H)
\]
whose composite with each projection $\pi_c$ is \ref{map_ind_by_proj_of_c}.

For each $\mathcal{V}$-functor $F:C\longrightarrow D$ and object $c\in C$,
let
\[
J^F_c:I\longrightarrow D(Fc,Fc)
\]
be the identity morphism
\[
J^F_c=j_{Fc}.
\]
The family $\{J^F_c\}_{c\in C}$ is dinatural, and therefore the universal
property of the end gives a unique morphism
\[
J_F:I\longrightarrow[C,D](F,F).
\]

These composition and identity morphisms equip $[C,D]$ with the structure of a
$\mathcal{V}$-category. Moreover, this construction defines the object part of
a functor
\[
[-,-]:
\mathcal{V}\mathbf{Cat}^{\mathrm{op}}
\times
\mathcal{V}\mathbf{Cat}
\longrightarrow
\mathcal{V}\mathbf{Cat}.
\]

Together with the tensor product $\otimes$ defined in the previous subsection,
the hom-functor $[-,-]$ equips $\mathcal{V}\mathbf{Cat}$ with the structure of
a symmetric monoidal closed category
\[
(\mathcal{V}\mathbf{Cat},\otimes,I,[-,-]).
\]

\subsection*{Enriched products}
To define Lawvere $\mathcal{V}$-theories in Section~5, we first need a notion
of products in a $\mathcal{V}$-category. Throughout this subsection, we assume
that $\mathcal{V}$ is cartesian closed. Although these definitions can be
developed in greater generality, this assumption is sufficient for our
purposes.
\begin{definition}
Let $C$ be a $\mathcal{V}$-category and let
$x_1,\ldots,x_n$ be objects of $C$. An \emph{$n$-fold enriched product} of
$x_1,\ldots,x_n$ consists of an object
\[
\prod_{i=1}^n x_i
\]
together with a $\mathcal{V}$-natural isomorphism
\[
C\!\left(-,\prod_{i=1}^n x_i\right)
\cong
\prod_{i=1}^n C(-,x_i).
\]
\end{definition}

In other words, an enriched product is characterized by the same universal
property as an ordinary product, with hom-sets replaced by hom-objects.

\begin{lemma}
An enriched product in a $\mathcal{V}$-category induces a product in the underlying category.
\end{lemma}

\begin{proof}
Let $x_1,\ldots,x_n$ be objects of $C$ and $y$ be an object of $C$. Then
\[
C_0(y,\prod_{i=1}^n x_i)=\mathcal{V}(I,C(y,\prod_{i=1}^n x_i))\cong \mathcal{V}(I,\prod_{i=1}^nC(y,x_i))\cong \prod_{i=1}^n\mathcal{V}(I,C(y,x_i))=\prod_{i=1}^nC_0(y, x_i),
\]
so that the enriched product in $C$ induces a product in the underlying category $C_0$.
\end{proof}

An enriched product need not be strict on the underlying category. Since our
later constructions require products that behave strictly on objects, we
record this additional property separately.

\begin{definition}
An enriched product is called \emph{strict} if its underlying product in the
ordinary category is strict.
\end{definition}

\begin{definition}
A  \emph{cartesian monoidal $\mathcal{V}$-category} is a $\mathcal V$-category $C$ such that
\begin{itemize}
    \item for all objects $x$ and $y$, we are given a choice of product $x\times y$, and
    \item there is a a choice of object $0$ with the property that for every object $x$, there is a unique morphism $!:*\to C(x,0)$ in $\mathcal V$.
\end{itemize}
\end{definition}

\section{Controlled theories}
\subsection*{The category of controlled theories}
We now introduce controlled theories. The guiding
idea is that a presentation of a Lawvere theory does not allow for a general method of categorification and homotopification. A
controlled theory enhances an ordinary presentation of a Lawvere theory by specifying a
distinguished portion of the free theory whose structure is retained and
tracked.
\begin{definition}\label{def_of_cont_th}
A \emph{controlled theory} consists of the following data:
\begin{itemize}
    \item a reduced signature $\mathcal{G}:\mathbf{N}\to\mathbf{Set}$,
    \item a pro $P$,
    \item a Lawvere theory $L$,
    \item a full morphism of Lawvere theories
    \[
    \mathbf{st}:\mathbf{Fr}(\mathcal{G})\to L,
    \]
    and
    \item a faithful morphism of pros
    \[
    i:P\to\mathbf{Fr}(\mathcal{G}).
    \]
\end{itemize}
\end{definition}

The reduced signature $\mathcal{G}$ specifies the generating operations, while
the morphism
\[
\mathbf{st}:\mathbf{Fr}(\mathcal{G})\to L
\]
encodes the relations imposed on these generators. The faithful morphism
\[
i:P\to\mathbf{Fr}(\mathcal{G})
\]
specifies the portion of the free theory whose structure data is being
tracked.

We denote a controlled theory of the form above by
\[
\left\langle
\mathcal{G}:
P\xrightarrow{i}\mathbf{Fr}(\mathcal{G})
\xrightarrow{\mathbf{st}}L
\right\rangle.
\]

\begin{definition}
A \emph{morphism of controlled theories}
\[
(I,f,H):
\left\langle
\mathcal{G}:P\xrightarrow{i}\mathbf{Fr}(\mathcal{G})
\xrightarrow{\mathbf{st}}L
\right\rangle
\longrightarrow
\left\langle
\mathcal{G}':P'\xrightarrow{i'}\mathbf{Fr}(\mathcal{G}')
\xrightarrow{\mathbf{st}'}L'
\right\rangle
\]
consists of a morphism of reduced signatures
\[
I:\mathcal{G}\to\mathcal{G}',
\]
a morphism of pros
\[
f:P\to P',
\]
and a morphism of Lawvere theories
\[
H:L\to L',
\]
such that the following diagram commutes.
\[
\begin{tikzpicture}[node distance=2cm]
\node (A) {$P$};
\node (B) [below of=A] {$P'$};
\node (C) [right of=A] {$\mathbf{Fr}(\mathcal{G})$};
\node (D) [below of=C] {$\mathbf{Fr}(\mathcal{G}')$};
\node (E) [right of=C] {$L$};
\node (F) [below of=E] {$L'$};

\draw[->] (A) -- node[left] {$f$} (B);
\draw[->] (A) -- node[above] {$i$} (C);
\draw[->] (B) -- node[below] {$i'$} (D);
\draw[->] (C) -- node[left] {$\mathbf{Fr}(I)$} (D);
\draw[->] (C) -- node[above] {$\mathbf{st}$} (E);
\draw[->] (D) -- node[below] {$\mathbf{st}'$} (F);
\draw[->] (E) -- node[right] {$H$} (F);
\end{tikzpicture}
\]
\end{definition}

We write $\mathbf{cTh}$ for the category of controlled theories. We now list some examples of controlled theories. 
\begin{example}\label{cont_theory_for_mon_gpds} We write $\Omega_{mon}:=\langle \mathcal{M}:\mathbf{N}[\mathcal{M}]\xrightarrow{i}\mathbf{Fr}(\mathcal{M})\xrightarrow{\mathbf{st}}\mathbf{M}\rangle$ to be the controlled theory that consists of
\begin{itemize}
\item  the signature $\mathcal{M}$ with one nullary operation and one binary operation,
\item the pro $\mathbf{N}[\mathcal{M}]$,
\item the Lawvere theory $\mathbf{M}$ for monoids,
\item the inclusion map $i:\mathbf{N}[\mathcal{M}]\to \mathbf{Fr}(\mathcal{M})$ of pros, and
\item the quotient map $\mathbf{st}:\mathbf{Fr}(\mathcal{M})\to\mathbf{M}$.
\end{itemize}
\end{example}

\begin{example}\label{cont_theory_for_sym_mon_gpds} We write $\Omega_{cm}:=\langle \mathcal{M}:\Sigma(\mathcal{M})\xrightarrow{i}\mathbf{Fr}(\mathcal{M})\xrightarrow{\mathbf{st}}\mathbf{CM}\rangle$ to be the controlled theory that consists of
\begin{itemize}
\item  the reduced signature $\mathcal{M}$ with one nullary operation and one binary operation,
\item the pro $\Sigma(\mathcal{M})$,
\item the Lawvere theory $\mathbf{CM}$ for commutative monoids,
\item the inclusion map $i:\Sigma(\mathcal{M})\to \mathbf{Fr}(\mathcal{M})$ of pros, and
\item the quotient map $\mathbf{st}:\mathbf{Fr}(\mathcal{M})\to\mathbf{CM}$.
\end{itemize}
\end{example}

\begin{example}\label{cont_theory_for_2_grps} We write $\Omega_{grp}:=\langle \mathcal{G}:\nabla(\mathcal{G})\xrightarrow{i}\mathbf{Fr}(\mathcal{G})\xrightarrow{\mathbf{st}}\mathbf{G}\rangle$ that consists of
  \begin{itemize}
 	\item  the reduced signature $\mathcal{G}$ with one nullary operation, one unary operation and one binary operation,
 	\item the pro $\nabla(\mathcal{G})$,
 	\item the Lawvere theory $\mathbf{G}$ for groups,
 	\item the inclusion map $i:\nabla(\mathcal{G})\to \mathbf{Fr}(\mathcal{G})$ of pros, and
 	\item the quotient map $\mathbf{st}:\mathbf{Fr}(\mathcal{G})\to\mathbf{G}$.
 \end{itemize}
\end{example}

These examples illustrate the purpose of the control pro. Rather than recording
only the generators and relations of a Lawvere theory, a controlled theory
retains a distinguished portion of the free theory whose structure data is
tracked throughout the construction. In this sense, these examples provide
a minimal presentation of algebraic structures together with the specific
coherence data that will be used in later constructions.

\subsection*{Diagrams of controlled theories} Our constructions we consider in higher categorical algebra are produced functorially; however, such constructions will not preserve colimits. We therefore need to work with special categories of diagrams to mediate this issue.
\begin{definition}
We define the category $\mathbf{Diag}(\mathbf{cTh})$ of diagrams in
$\mathbf{cTh}$ as follows. An object consists of a pair $(I,D)$, where $I$ is
a small category and
\[
J:I\to\mathbf{cTh}
\]
is a diagram of controlled theories. Given two objects $(I,D)$ and
$(I',D')$, define
\[
\mathbf{Diag}(\mathbf{cTh})((I,J),(I',J'))
=
\begin{cases}
[I,\mathbf{cTh}](J,J') & \text{if } I=I',\\
\emptyset & \text{otherwise}.
\end{cases}
\]
\end{definition}

We will primarily consider diagrams of controlled theories whose indexing categories are connected.

\begin{notation}
We write $\mathbf{Diag}^{+}(\mathbf{cTh})$ for the full subcategory of
$\mathbf{Diag}(\mathbf{cTh})$ consisting of connected diagrams.
\end{notation}

We now include the following lemma to motivate the use of connected diagrams of controlled
theories.
\begin{lemma}\label{ab_is_pullback}
The category of abelian groups is the pullback in the following diagram.
\[
\begin{tikzpicture}[node distance=3cm]
\node (A) {$\mathbf{Ab}$};
\node (B)[right of=A] {$\mathbf{Grp}$};
\node (C)[below of=A] {$\mathbf{cMon}$};
\node (D)[right of=C] {$\mathbf{Mon}$};

\draw[->] (A) -- node[above] {$U$} (B);
\draw[->] (A) -- node[left] {$U$} (C);
\draw[->] (B) -- node[right] {$U$} (D);
\draw[->] (C) -- node[below] {$U$} (D);
\end{tikzpicture}
\]
\end{lemma}

\begin{proof}
An abelian group consists of both a group structure and a commutative monoid
structure on the same underlying set, such that the two structures agree after
forgetting to a monoid. Hence the category of abelian groups is the pullback of
the categories of groups and commutative monoids over the category of monoids.
\end{proof}

We are motivated by this lemma to consider the following example of controlled theories. 

\begin{example}\label{conn_diag_cont_theory_for_Pic}
There is a connected diagram of controlled theories, which we denote by
$\Gamma^b_{pic}$, given by
\[
\begin{tikzpicture}[node distance=3cm,auto]
\node (A) {$\Omega_{mon}$};
\node (B)[right of=A] {$\Omega_{cm}$};
\node (C)[below of=A] {$\Omega_{grp}$};

\draw[->] (A) -- node[above] {$i$} (B);
\draw[->] (A) -- node[left] {$j$} (C);
\end{tikzpicture}
\]
where $\Omega_{mon}$, $\Omega_{cm}$, and $\Omega_{grp}$ are the controlled
theories defined in Examples~\ref{cont_theory_for_mon_gpds},
\ref{cont_theory_for_sym_mon_gpds}, and \ref{cont_theory_for_2_grps},
respectively.
\end{example}

\subsection*{Pros from controlled theories}
Let
\[
\Omega:=\langle \mathcal{G}:P\xrightarrow{i}\mathbf{Fr}(\mathcal{G})
\xrightarrow{\mathbf{st}}L\rangle
\]
be a controlled theory. The control pro $P$ contains more information than is
present in the resulting Lawvere theory $L$. In particular, different
morphisms in $P$ may determine the same morphism after applying the composite of the structural
maps.  Let $n,k\geq 0$ and $f\in P(n,k)$. We write 
\[
[f]=\{g\in P(n,k):\mathbf{st}(i(g))=\mathbf{st}(i(f))\}.
\]
The \emph{reduced pro} on the controlled theory $\Omega$, denoted by $\overline{\Omega}$, has homset 
\[
\overline{\Omega}(n,k)=\{[f]:f\in P(n,k)\}
\]
for all $n,k\geq 0$. The reduced pro records precisely the portion of the
control data that remains visible in the associated Lawvere theory. Moreover, the operations that we want to identify in higher dimensional constructions are put together in one set.

 \begin{example}
The controlled theory from Example \ref{cont_theory_for_mon_gpds} has reduced pro whose algebras are monoids as its reduced pro. 
 \end{example}

 \begin{lemma}\label{reduced_pro}
The construction of the induced pro from a controlled theory extends to a functor 
\[
\overline{(-)}:\mathbf{cTh}\to\mathbf{Pro}.
\]
 \end{lemma}

 \begin{proof}
Let 
\[
	\Phi:=(I,f,H):\langle \mathcal{G}:P\xrightarrow{i}\mathbf{Fr}(\mathcal{G})\xrightarrow{\mathbf{st}}L\rangle\to \langle \mathcal{G}':P'\xrightarrow{i'}\mathbf{Fr}(\mathcal{G}')\xrightarrow{\mathbf{st}'}L'\rangle
    \] 
be a map of controlled theories. Define
\[
\Phi([\alpha])=[f(\alpha)]
\]
for all $n,k\geq 0$ and $[\alpha]\in \overline{\Omega}(n,k)$. By construction, this is functorial.
 \end{proof}

\subsection*{Admissible controlled theories}
The three controlled theories that we consider in this paper and the controlled theories that we consider in our next paper have a special quality to them. We now rigorously give a description for that quality.
\begin{construction}\label{pullback_pro}
Let $\Omega:=\langle \mathcal{G}:P\xrightarrow{i}\mathbf{Fr}(\mathcal{G})\xrightarrow{\mathbf{st}}L\rangle$ be a controlled theory. The \emph{pullback pro} of $\Omega$, denoted by $\Omega^*$, is defined to be the pullback of the following diagram in $\mathbf{Pro}$.
\[
  \begin{tikzpicture}[node distance=2cm, auto]
        \node (A) {};
       \node (B)[right of=A]{$\mathcal{F}$};
       \node (C)[below of=A]{$P$};
       \node(D)[right of=C]{$\mathbf{Fr}(\mathcal{G})$};
       \draw[->](B) to node[right=3]{$!$} (D);
       \draw[->](C) to node[below=3]{$i$} (D);
    \end{tikzpicture}
\]
This pro comes equipped with a faithful map of pros $i_\Omega:\Omega^*\to \mathcal{F}$.
\end{construction}

\begin{lemma}
Construction \ref{pullback_pro} extends to a functor of the following form.
\[
((-)^*,i_{(-)}):\mathbf{cTh}\to\mathbf{Pro}/\mathcal{F}
\]
\end{lemma}

    Suppose that $P$ is a pro equipped with a faithful map of pros $i:P\to\mathcal{F}$. Let $\mathcal{G}$ be a reduced signature. There is a naturally induced map $P[\mathcal{G}]\xrightarrow{i_\mathcal{G}}\mathbf{Fr}(\mathcal{G})$ which induces a pro $P(\mathcal{G})$ together with a factorization 
    \[
    P[\mathcal{G}]\xrightarrow{q}P(\mathcal{G})\xrightarrow{i^\mathcal{G}}\mathbf{Fr}(\mathcal{G})
    \]
    such that $i^\mathcal{G}$ is a faithful map of pros and $q(f)=q(g)$ if and only if $i_\mathcal{G}(f)=i_\mathcal{G}(g)$. 

    \begin{definition}
    A controlled theory $\Omega:=\langle \mathcal{G}:P\xrightarrow{i}\mathbf{Fr}(\mathcal{G})\xrightarrow{\mathbf{st}}L\rangle$ is \emph{admissible} if 
    \[
    \Omega^*(G)\cong P
    \]
    as pros.
    \end{definition}

\begin{lemma}
    An admissible controlled theory can be identified with a triple consisting of a pro $P$ equipped with a faithful map of pros $i:P\to\mathcal{F}$ together with a reduced signature $G$ and a congruence relation $\mathcal{R}$ on $\mathbf{Fr}(G)$.
\end{lemma}

This provides us a new definition of admissible controlled theory.

\begin{definition}
An \emph{admissible controlled theory} consists of a pro $P$ equipped with a faithful map of pros $i:P\to\mathcal{F}$ together with a reduced signature $G$ and a congruence relation $\mathcal{R}$ on $\mathbf{Fr}(G)$.

We write $(P\xrightarrow{i}\mathcal{F},G,\mathcal{R})$ to denote an admissible controlled theory.
\end{definition}

\begin{definition}
A morphism of admissible controlled theories from 
\[
(P\xrightarrow{i}\mathcal{F},G,\mathcal{R})
\]
to 
\[
(Q\xrightarrow{j}\mathcal{F},H,\mathcal{R}')
\]
is a pair $(f,k)$ where $f:G\to H$ is a map of reduced signatures and $k:P\to Q$ is a map of pros such that $j\circ k=i$ and $\mathbf{Fr}(f)(\mathcal{R})\subset \mathcal{R}'$.
\end{definition}

We write $\mathbf{acTh}$ to denote the category of admissible controlled theories. The following proposition follows from using the locally presentable nature of reduced signatures, pros, and Lawvere theories.

\begin{proposition}
The category of admissible controlled theories has all small limits and colimits, i.e., is bicomplete as a category.
\end{proposition}

We want to now show how every controlled theory has a suitable replacement to an admissible controlled theory. Let 
\[
\Omega:=\langle \mathcal{G}:P\xrightarrow{i}\mathbf{Fr}(\mathcal{G})\xrightarrow{\mathbf{st}}L\rangle
\]
be a controlled theory. Define $R(\Omega)=(\Omega^*\xrightarrow{i_\Omega}\mathcal{F},\mathcal{G},R_\Omega)$ where $R_\Omega$ is the congruence relation generated by the following map of Lawvere theories.
\[
\mathbf{st}:\mathbf{Fr}(\mathcal{G})\to L
\]
This upgrades to a functor $R:\mathbf{cTh}\to \mathbf{acTh}$. 

Let $(P\xrightarrow{i}\mathcal{F},G,\mathcal{R})$ be an admissible controlled theory. Define 
\[
S(P\xrightarrow{i}\mathcal{F},G,\mathcal{R})=\langle G:P(G)\xrightarrow{i^G}\mathbf{Fr}(G)\xrightarrow{\mathbf{st}}\mathbf{Fr}(\mathcal{G})/\mathcal{R}\rangle.
\]
This upgrades to a functor
\[
S:\mathbf{acTh}\to\mathbf{cTh}.
\]
The first thing to notice is that there is a natural isomorphism $\epsilon:R\circ S\cong 1_{\mathbf{acTh}}$. Moreover, there is a natural transformation $\eta:1_{\mathbf{cTh}}\Rightarrow S\circ R$. The data $(S,R,\eta,\epsilon)$ forms an adjunction. The \emph{admissible replacement} of a controlled theory $\Omega$ is the controlled theory $S(R(\Omega))$ and it comes equipped with a morphism
\[
\eta_{\Omega}:\Omega\to S(R(\Omega))
\]
of controlled theories.

We shall not use this adjunction or the admissible replacement in this paper. We include it here because it is interesting and we will require it in future work.

 \section{Enriched categorical algebra}
 For this section, we fix a locally presentable cartesian closed category $\mathcal{V}$ with a choice of arbitrarily indexed coproducts.
\subsection*{Enriched signatures}
 \begin{definition}
    A \emph{$\mathcal{V}$-signature} is a functor $\chi:\mathbf{N}^2\to \mathcal{V}$.
 \end{definition}

We write $\mathcal{V}\mathbf{Sig}$ to denote the functor category $[\mathbf{N}^2,\mathcal{V}]$. Let $\chi:\mathbf{N}^2\to\mathcal{V}$ be a $\mathcal{V}$-signature. Its \emph{underlying signature} $U(\chi)$ is defined by
\[
U(\chi)_{n,m}=\mathcal{V}(I,\chi_{n,m})
\]
for all $n,m\geq 0$. This upgrades to a functor $U:\mathcal{V}\mathbf{Sig}\to\mathbf{Sig}$. Moreover, $U$ has a left adjoint, call it $D$.

\subsection*{Enriched pros}
\begin{definition}
A \emph{$\mathcal{V}$-pro} consists of a $\mathcal{V}$-signature $P$ together with
\begin{itemize}
    \item a \emph{unit morphism} $1_k:*\to P(k,k)$ in $\mathcal{V}$ for all $k\geq 0$, and
    \item a \emph{composition morphism} 
    \[
    \Gamma:P(\sum_{i=1}^tm_i,k)\times\prod_{i=1}^tP(n_i,m_i)\to P(\sum_{i=1}^tn_i,k)
    \]
    in $\mathcal{V}$ for all $t\geq 1$ and $n_1,m_1,\dots n_t,m_t,k\geq 0$.
\end{itemize}
This data satisfies axioms of compatibility of unities, unital composition, and associative composition.
\end{definition}

The characterization we have provided in this definition is similar to the characterization of enriched monad provided in \cite{McDermottUustalu2022}. We did not specify the structure of a $\mathcal{V}$-category. The $\mathcal{V}$-category structure exists by the additonal structure and axioms. We leave the following proposition to the reader. 

\begin{proposition}
A $\mathcal{V}$-pro is the same as a monoid in $\mathcal{V}\mathbf{Cat}$ with object set $\mathbf{N}$, where we write $\mathcal{V}\mathbf{Cat}$ to mean the underlying category of the $2$-category with the same name.
\end{proposition}

\begin{example}
A $\mathbf{Set}$-pro is a pro in the standard sense.  
\end{example}


\begin{definition}
Let $(P,1,\Gamma)$ and $(P',1',\Gamma')$ be $\mathcal{V}$-pros. A map $f:P\to P'$ of $\mathcal{V}$-pros consists of a morphism $f_{n,k}:P(n,k)\to P'(n,k)$ for all $n,k\geq 0$. This data satisfies axioms for preservation of units and preservation of composition.
\end{definition}

We write $\mathcal{V}\mathbf{Pro}$ to denote the category of $\mathcal{V}$-pros.

\begin{proposition}
The underlying category of a $\mathcal{V}$-pro has the structure of a pro, i.e., the functor $U:\mathcal{V}\mathbf{Cat}\to\mathbf{Cat}$ lifts to a functor
$U:\mathcal{V}\mathbf{Pro}\to \mathbf{Pro}$. Moreover, $U$ has a left adjoint.
\end{proposition}

\begin{proof}
The first part of the proposition is left to the reader as an exercise. We proceed with proving that $U$ has a left adjoint.

Let $P$ be a pro. Define $D(P)(n,m)=\coprod_{f\in P(n,m)}*$ for all $n,m\geq 0$ and define $1_k:*\to D(P)(1,1)$ as $1_k=\iota_{\text{id}_k}$ for all $k\geq 0$. Let $t\geq 1$ and $n_1,\dots, n_t,m_1,\dots, m_t,k\geq 0$. Define the composition map as the following composite.
\[
\begin{tikzpicture}[node distance=2cm]
    \node (A) {$D(P)(\sum_{i=1}^tm_i,k)\times\prod_{i=1}^tD(P)(n_i,m_i)$};
    \node (B)[below of=A]{$\coprod_{g\in P(\sum_{i=1}^tm_i,k)}*\times \prod_{i=1}^t\coprod_{f_i\in P(n_i,m_i)}*$};
    \node (C) [below of=B]{$\coprod_{g\in P(\sum_{i=1}^tm_i,k)}\coprod_{f_1\in P(n_1,m_1)}\dots \coprod_{f_t\in P(n_t,m_t)}(*\times\prod_{i=1}^t*)$};
    \node (D)[below of=C]{$\coprod_{g\in P(\sum_{i=1}^tm_i,k)}\coprod_{f_1\in P(n_1,m_1)}\dots \coprod_{f_t\in P(n_t,m_t)}*$};
    \node (E)[below of=D]{$D(P)(\sum_{i=1}^tn_i,k)$};
    \draw[->](A) to node[left=3]{$=$}(B);
    \draw[->](B) to node[left=3]{$\cong$}(C);
    \draw[->](C) to node[left=3]{$\cong$}(D);
    \draw[->](D) to node[left=3]{$\iota_{g\circ\sum_{i=1}^tf_i}$}(E);
\end{tikzpicture}
\]
The first isomorphism is the distributivity isomorphism induced from $\mathcal{V}$ being cartesian closed. The second isomorphism is a composite of unitor isomorphisms for $\mathcal{V}$. The final map is a structure map for the coproducts. The triple $(D(P),1,\Gamma)$ determines a $\mathcal{V}$-pro. Moreover, this upgrades to a functor $D:\mathbf{Pro}\to\mathcal{V}\mathbf{Pro}$.

Let $P$ be a pro. Define a map $\eta^P:P\to U(D(P))$ of pros by
\[
\eta^P_{n,m}(f)=\iota_f.
\]
The maps provide a natural transformation
\[
\eta:1_{\mathbf{Pro}}\Rightarrow U\circ D.
\]

Let $P$ be a $\mathcal{V}$-pro. Define $\epsilon^P:D(U(P))\to P$ by
\[
\epsilon^P_{n,m}=\langle f\rangle_{f\in U(P)(n,m)}
\]
where $\langle f\rangle_{f\in U(P)(n,m)}$ is induced by the universal property of the coproduct. The maps provide a natural transformation
\[
\epsilon:D\circ U\Rightarrow 1_{\mathcal{V}\mathbf{Pro}}.
\]
This data constitutes our desired adjunction.
\end{proof}

We now briefly define the initial object of $\mathcal{V}\mathbf{Pro}$, which we choose to denote by $\mathbf{N}_{\mathcal{V}}$. The $\mathcal{V}$-pro $\mathbf{N}_\mathcal{V}$ has hom-object 
\[
\mathbf{N}_\mathcal{V}(n,m)=\begin{cases}
    *\quad\text{if } $n=m$\\
    \emptyset\quad{otherwise.}
\end{cases}
\]
The definition of $\mathcal{V}$-pro guarantees the existence of a unique map $!:\mathbf{N}_\mathcal{V}\to P$ for every $\mathcal{V}$-pro $P$.

\begin{notation}\label{und_V_sig_fun}
  Let $P:=(P,1,\Gamma)$ be a $\mathcal{V}$-pro. The \emph{underlying $\mathcal{V}$-signature} functor, which we denote by $U:\mathcal{V}\mathbf{Pro}\to\mathcal{V}\mathbf{Sig}$, is defined by
\[
U(P)_{n,m}=P(n,m).
\]  
\end{notation}

\begin{lemma}\label{left_adjoint_for_V_Pro_exists}
 The functor $U$ of Notation \ref{und_V_sig_fun} has a left adjoint.
\end{lemma}

\begin{proof}
Since $\mathcal{V}$ is locally presentable, then $\mathcal{V}\mathbf{Sig}$ and $\mathcal{V}\mathbf{Pro}$ are. A quick check shows that $U$ preserves limits and filtered colimits. Therefore, $U$ has a left adjoint by the Adjoint Functor Theorem for locally presentable categories.
\end{proof}

 We shall call the left adjoint $\mathbf{N}_{\mathcal{V}}[-]$. We assigned this naming convention with the naming convention of the initial object. 
 
 \begin{remark}
     We remark that it is possible to construct the left adjoint when $\mathcal{V}$ is not locally presentable. The prime example that comes to mind is $\mathbf{Top}$. 
 \end{remark}

\begin{construction}
Let $C$ be a $\mathcal{V}$-category and $x$ be an object of $C$ with a choice of all finite iterated enriched products of $x$ with itself, denoted by $x^n$ for all $n\geq 0$. Define $\mathbf{End}_C(x)$ to have object set $\mathbf{N}$ and the hom-object to be defined by $\mathbf{End}_C(x)(n,m)=C(x^n,x^m)$ for all $n,m\geq 0$. As we allude to in the next subsection, $\mathbf{End}_C(x)$ maybe equipped with the structure of a Lawvere $\mathcal{V}$-theory $(\mathbf{End}_C(x),J_x)$ and therefore has a $\mathcal{V}$-pro structure. Additionally, we obtain a $\mathcal{V}$-functor $i_x:\mathbf{End}_C(x)\to C$ that preserves finite enriched products.
\end{construction}

\begin{definition}
Let $C$ be a $\mathcal{V}$-category, $x$ be an object of $C$ with a choice of all finite iterated enriched products of $x$ with itself, and $P$ be a $\mathcal{V}$-pro. An \emph{algebra structure} on $x$ consists of a map of pros $A:P\to \mathbf{End}_C(x)$. An \emph{algebra} of $P$ in $C$ consists of an object $x$ with a choice of all finite iterated enriched products of $x$ with itself together with an algebra structure $\phi:P\to \mathbf{End}_C(x)$ on $x$.
\end{definition}

\begin{definition}
Let $C$ be a $\mathcal{V}$-category and $(x,\phi)$ and $(y,\lambda)$ be algebras of $P$ in $C$. A morphism $f:(x,\phi)\to(y,\lambda)$ of algebras consists of a morphism $f:*\to C(x,y)$ such that the following diagram commutes.
\[
\begin{tikzpicture}[node distance=3cm]
    \node (A) {$P(n,m)$};
    \node (B)[right of=A]{$C(x^n,x^m)$};
    \node (C)[below of=A]{$C(y^n,y^m)$};
    \node (D)[right of=C]{$C(x^n,y^m)$};
    \draw[->](A) to node[above=3]{$\phi_{n,m}$}(B);
    \draw[->](A) to node[left=3]{$\lambda_{n,m}$}(C);
    \draw[->](B) to node[right=3]{$(f^m)^*$}(D);
    \draw[->](C) to node[below=3]{$f^n_*$}(D);
\end{tikzpicture}
\]
\end{definition}

We write $\mathbf{Alg}(P,C)$ to denote the category of algebras of $P$ in $C$. The composition of morphisms and identities are obtained from the underlying category of $C$.

\subsection*{Enriched Lawvere theories}

\begin{definition}
Define $\mathcal{F}_V$ to be the opposite $\mathcal{V}$-category of the $\mathcal{V}$-category with object set $\mathbf{N}$ and whose hom-objects are defined by 
\[
\mathcal{F}^\text{op}_\mathcal{V}(n,m)=[n_\mathcal{V},m_{\mathcal{V}}],
\]
where $[-,-]$ is the internal hom of $\mathcal{V}$.
\end{definition}

\begin{lemma}
For all $n\geq 0$, $n$ is the strict $n$-fold product of $1$ in $\mathcal{F}_\mathcal{V}$, so that  $\mathcal{F}_\mathcal{V}$ is a cartesian monoidal $\mathcal{V}$-category with strict products.
\end{lemma}

\begin{proof}
This follows directly from Lemma \ref{ind_fun_iso_to_prod}.  
\end{proof}

\begin{definition}
    A \emph{Lawvere $\mathcal{V}$-theory} consists of a $\mathcal{V}$-category $L$ together with a bijective-on-objects $\mathcal{V}$-functor $J:\mathcal{F}_\mathcal{V}\to L$ that preserves products.
\end{definition}

\begin{example}
A Lawvere $\mathbf{Cat}$-theory is a Lawvere $2$-theory.
\end{example}

\begin{example}
Let $C$ be a $\mathcal{V}$-category and $x$ be an object of $C$ with a choice of all finite iterated enriched products of $x$ with itself, denoted by $x^n$ for all $n\geq 0$. Then $\mathbf{End}_C(x)$ is a Lawvere $\mathcal{V}$-theory.
\end{example}

\begin{definition}
Let $L$ and $L'$ be Lawvere $\mathcal{V}$-theories. A map of Lawvere $\mathcal{V}$-theories from $(L,J)$ to $(L',J')$ consists of an enriched product-preserving $\mathcal{V}$-functor $F:L\to L'$  such that the following diagram commutes. 
    \[
    \begin{tikzpicture}[node distance=3cm]
        \node (A) {$\mathcal{F}_\mathcal{V}$};
        \node (B)[right of=A]{$L$};
        \node (C)[below of=B]{$L'$};
        \draw[->](A) to node[above] {$J$} (B);
        \draw[->](A) to node[left] {$J'$} (C);
        \draw[->](B) to node[right]{$F$}(C);
    \end{tikzpicture}
    \]
    
\end{definition}

We write $\mathcal{V}\mathbf{Law}$ to denote the category of Lawvere $\mathcal{V}$-theories.

Let $(L,J)$ be a Lawvere $\mathcal{V}$-theory. Define $U(L,J)=(L,1,\Gamma)$. We now define the structure maps. Define $1_k:*\to L(k,k)$ by $1_k:=j_k$ for all $k\geq 0$. Let $t\geq 1$ and $n_1,m_1,\dots n_t,m_t,k\geq 0$. We first define a map $\pi'_j:*\to L(\sum_{i=1}^tn_i,n_j)$ as the following composite for all $1\leq j\leq t$.
\[
*\xrightarrow{j_{\sum_{i=1}^tn_i}}L(\sum_{i=1}^tn_i,\sum_{i=1}^tn_i)\cong \prod_{i=1}^tL(\sum_{i=1}^tn_i,n_i)\xrightarrow{\pi_j}L(\sum_{i=1}^tn_i,n_j)
\]
Furthermore, define a morphism $\phi_j:\prod_{i=1}^tL(n_i,m_i)\to L(\sum_{i=1}^tn_i,m_j)$ as the following composite for all $1\leq j\leq t$.
\[
\prod_{i=1}^tL(n_i,m_i)\xrightarrow{\pi_j}L(n_j,m_j)\xrightarrow{(\pi'_j)_*}L(\sum_{i=1}^tn_i,m_j)
\]
By applying the universal property of the product and the structure map for enriched products, we obtain a morphism 
\[
+:\prod_{i=1}^tL(n_i,m_i)\to L(\sum_{i=1}^tn_i,\sum_{i=1}^tm_i).
\]
Now define 
\[
\Gamma:L(\sum_{i=1}^tm_i,k)\times\prod_{i=1}^tL(n_i,m_i)\to L(\sum_{i=1}^tn_i,k)
\]
as the following composite.
\[
L(\sum_{i=1}^tm_i,k)\times\prod_{i=1}^tL(n_i,m_i)\xrightarrow{1\times +}L(\sum_{i=1}^tm_i,k)\times L(\sum_{i=1}^tn_i,\sum_{i=1}^tm_i)\xrightarrow{\circ}L(\sum_{i=1}^tn_i,k)
\]
The data $U(L,J)$ is a $\mathcal{V}$-pro. This gives rise to a functor $U:\mathcal{V}\mathbf{Law}\to\mathcal{V}\mathbf{Pro}$.
 
\begin{proposition}
The functor $U:\mathcal{V}\mathbf{Law}\to\mathcal{V}\mathbf{Pro}$ has a left adjoint.
\end{proposition}

\begin{proof}
Since $\mathcal{V}$ is a locally presentable cartesian closed category, $\mathcal{V}\mathbf{Pro}$ and $\mathcal{V}\mathbf{Law}$ are locally presentable. The forgetful functor $U$ preserves limits and filtered colimits. Therefore $U$ has a left adjoint.
\end{proof}

We shall write $L_{(-)}^\mathcal{V}:\mathcal{V}\mathbf{Pro}\to\mathcal{V}\mathbf{Law}$ for the left adjoint.

\begin{proposition}
The underlying category of a Lawvere $\mathcal{V}$-theory has the structure of a Lawvere theory, i.e., the functor $U:\mathcal{V}\mathbf{Pro}\to\mathbf{Pro}$ lifts to a functor $U:\mathcal{V}\mathbf{Law}\to\mathbf{Law}$.  If the functor of Lemma \ref{copies_functor} is strong monoidal, then $U$ has a left adjoint.
\end{proposition}

\begin{definition}
Let $C$ be a $\mathcal{V}$-category, $x$ be an object of $C$ with a choice of all finite iterated enriched products of $x$ with itself, and $L$ be a Lawvere $\mathcal{V}$-theory. An \emph{algebra structure} on $x$ consists of a map of Lawvere $\mathcal{V}$-theories $A:L\to \mathbf{End}_C(x)$. An \emph{algebra} of $L$ in $C$ consists of an object $x$ with a choice of all finite iterated enriched products of $x$ with itself together with an algebra structure $\phi:L\to \mathbf{End}_C(x)$ on $x$.
\end{definition}

\begin{definition}
Let $C$ be a $\mathcal{V}$-category and $(x,\phi)$ and $(y,\lambda)$ be algebras of $L$ in $C$. A morphism $f:(x,\phi)\to(y,\lambda)$ of algebras consists of a morphism $f:*\to C(x,y)$ such that the following diagram commutes.
\[
\begin{tikzpicture}[node distance=3cm]
    \node (A) {$L(n,m)$};
    \node (B)[right of=A]{$C(x^n,x^m)$};
    \node (C)[below of=A]{$C(y^n,y^m)$};
    \node (D)[right of=C]{$C(x^n,y^m)$};
    \draw[->](A) to node[above=3]{$\phi_{n,m}$}(B);
    \draw[->](A) to node[left=3]{$\lambda_{n,m}$}(C);
    \draw[->](B) to node[right=3]{$(f^m)^*$}(D);
    \draw[->](C) to node[below=3]{$f^n_*$}(D);
\end{tikzpicture}
\]
\end{definition}

\begin{theorem}
Let $J:I\to\mathcal{V}\mathbf{Law}$ be a connected diagram of Lawvere $\mathcal{V}$-theories. Then the category of algebras over $\text{colim}_{i\in\mathbf{ob}(I)}J(i)$ in $\underline{V}$ is isomorphic to the limit of the following diagram.
\[
I^\text{op}\xrightarrow{J^\text{op}}\mathcal{V}\mathbf{Law}^\text{op}\xrightarrow{\mathbf{Alg}(-,\underline{V})}\mathbf{CAT}
\]
\end{theorem}

\begin{proof}
This generalizes Theorem \ref{Alg_pre_conn_colimits} and follows directly from Example 2.10 of \cite{ArkorMcDermott2025}.
\end{proof}

We naturally obtain a forgetful functor 
 \[
 U:\mathbf{Alg}(L,\underline{\mathcal{V}})\to \mathcal{V}
 \]
 and its left adjoint 
 \[
 \mathbf{Fr}_L:\mathcal{V}\to \mathbf{Alg}(L,\underline{\mathcal{V}})
 \]
 defined on objects with the following formula that implements a coend.
 \[
 \mathbf{Fr}_L(v)=\left(\int^{n\geq 0}L(n,1)\times v^{\times n}, \gamma^{\mathbf{Fr}_L(v)}\right)
 \]
 We use the fact that $\mathcal{V}$ is cartesian monoidal closed to define the structure maps $\gamma^{\mathbf{Fr}_L(v)}$. The following theorem holds. 

 \begin{theorem}
Given a Lawvere $\mathcal{V}$-theory, the adjunction
\[
\begin{tikzcd}[column sep=large]
\mathcal{V}
\arrow[r, shift left=1.2ex, "\mathbf{Fr}_L"]
&
\mathbf{Alg}(L,\underline{\mathcal{V}})
\arrow[l, shift left=1.2ex, "U"]
\end{tikzcd}
\]
is monadic. 
 \end{theorem}

Now let $f:L\to L'$ be a map of Lawvere $\mathcal{V}$-theories and $C$ be a $\mathcal{V}$-category. There is an induced map of algebras
\[
f_*:\mathbf{Alg}(L',C)\to \mathbf{Alg}(L,C)
\]
obtained by pre-composing with $f$. Let $T_L$ and $T_{L'}$ be the induced monads for $L$ and $L'$, respectively. Let $f^*:T_L\to T_{L'}$ be the induced map of monads from $f$.

We now describe the left adjoint to the pre-composition functor on objects when we restrict our attention to $\underline{\mathcal{V}}$. Let $(X,\alpha^X)$ be an algebra of $L$ considered as an algebra over the monad $T_L$. We induce a reflexive pair of $T_{L'}$-algebras and thus algebras of $L'$.

\[
\begin{tikzpicture}[>=latex]
\node (A) at (0,0) {$T_{L'}(T_L(X))$};
\node (B) at (5,0) {$T_{L'}(X)$};

\draw[->]
([yshift=3pt]A.east) --
node[above] {$\mu^{L'}_X\circ f^*(X)$}
([yshift=3pt]B.west);

\draw[->]
([yshift=-3pt]A.east) --
node[below] {$T_{L'}(\alpha_X)$}
([yshift=-3pt]B.west);
\end{tikzpicture}
\]

We now define $f_!(X,\alpha^X)$ to be the coequalizer of this diagram. This construction upgrades to a functor $f_!:\mathbf{Alg}(L,\underline{\mathcal{V}})\to \mathbf{Alg}(L',\underline{\mathcal{V}})$ by the universal property of the coequalizer and is the left adjoint to $f_*:\mathbf{Alg}(L',\underline{\mathcal{V}})\to \mathbf{Alg}(L,\underline{\mathcal{V}})$. 

\subsection*{Deformations over a pro}

\begin{definition}
    Let $P$ be a pro. A \emph{deformation} of $P$ in $\mathcal{V}$ consists of \begin{itemize}
    \item an object $\mathcal{O}_{f}$ for all $n,k\geq 0$ and $f\in P(n,k)$,
    \item a morphism
    \[
    \eta_k:*\to \mathcal{O}_{1_k},
    \]
    called the \emph{$k$-fold unity map}, for all $k\geq 0$, and
    \item a morphism
    \[
    \Gamma:\mathcal{O}_{g}\times\prod_{i=1}^t\mathcal{O}_{f_i}\to\mathcal{O}_{g\circ\sum_{i=1}^tf_i},
    \]
   called \emph{composition}, for all $t\geq 1$, $n_1,\dots,n_t,m_1,\dots,m_t,k\geq 0$, $f_i\in P(n_i,m_i)$ for $1\leq i\leq t$, and $g\in P(\sum_{i=1}^tm_i,k)$.
   
\end{itemize}
This data satisfies the following axioms.
\begin{itemize}
    \item \emph{Associativity of composition for parametrization objects.} The following diagram commutes for all 
compatible indices.
    \[
    \begin{tikzpicture}[node distance=3cm, auto]
        \node (A) {$\mathcal{O}_{g}\times \prod_{i=1}^t(\mathcal{O}_{f_i}\times\prod_{j=1}^{n_i}\mathcal{O}_{h_{j,i}})$};
        \node (B)[below of=A, node distance=2cm]{};
        \node (C)[left of=B, node distance=4cm]{$\mathcal{O}_{g}\times \prod_{i=1}^t\mathcal{O}_{f_i}\times \prod_{i=1}^t\prod_{j=1}^{n_i}\mathcal{O}_{h_{j,i}}$};
        \node (D)[right of=B, node distance=4cm]{$\mathcal{O}_{g}\times \prod_{i=1}^t\mathcal{O}_{f_i\circ \sum_{j=1}^{n_i}h_{j,i}}$};
        \node (E)[below of=B, node distance=2cm]{};
        \node (F)[left of=E, node distance=3cm]{$\mathcal{O}_{g\circ \sum_{i=1}^tf_i}\times \prod_{i=1}^t\prod_{j=1}^{n_i}\mathcal{O}_{h_{j,i}}$};
        \node (G)[right of=E, node distance=3cm]{$\mathcal{O}_{g\circ \sum_{i=1}^t(f_i\circ\sum_{j=1}^{n_i}h_{j,i})}$};
        \draw[->](A) to node[left=3] {$\cong$}(C);
        \draw[->](A) to node[right=3] {$1\times\prod_{i=1}^t\Gamma$}(D);
        \draw[->](C) to node[left=3] {$\Gamma\times 1$}(F);
        \draw[->](F) to node[below=3] {$\Gamma$}(G);
        \draw[->](D) to node[right=3] {$\Gamma$}(G);
    \end{tikzpicture}
\]
\item \emph{Regulation of Unities.} The following diagram commutes for all 
compatible indices.
\[
  \begin{tikzpicture}[node distance=3cm, auto]
        \node (A) {$*$};
       \node (B)[right of=A]{$*\times\prod_{i=1}^t*$};
       \node (C)[right of=B,node distance=5cm]{$\mathcal{O}_{1_{\sum_{i=1}^tn_i}}\times\prod_{i=1}^t\mathcal{O}_{1_{n_i}}$};
       \node (D)[below of=C]{$\mathcal{O}_{1_{\sum_{i=1}^tn_i}}$};
       \draw[->] (A) to node[above=3]{$\cong$}(B);
       \draw[->] (B) to node[above=3]{$\eta_{\sum_{i=1}^tn_i}\times \prod_{i=1}^t\eta_{n_i}$}(C);
       \draw[->](C) to node[right=3]{$\Gamma$}(D);
       \draw[->](A) to node[below=3]{$\eta_{\sum_{i=1}^tn_i}$} (D);
    \end{tikzpicture}
\]
\item \emph{Left Unity.} The following diagram commutes for all 
compatible indices.
\[
  \begin{tikzpicture}[node distance=3cm, auto]
        \node (A) {$\mathcal{O}_{f}$};
       \node (B)[right of=A]{$*\times\mathcal{O}_{f}$};
       \node (C)[right of=B]{$\mathcal{O}_{1_k}\times\mathcal{O}_{f}$};
       \node (D)[below of=C]{$\mathcal{O}_{f}$};
       \draw[->] (A) to node[above=3]{$\cong$}(B);
       \draw[->] (B) to node[above=3]{$\eta_{k}\times 1$}(C);
       \draw[->](C) to node[right=3]{$\Gamma$}(D);
       \draw[->](A) to node[below=3]{$1$} (D);
    \end{tikzpicture}
\]
\item \emph{Right Unity.} The following diagram commutes for all 
compatible indices.
\[
  \begin{tikzpicture}[node distance=3cm, auto]
        \node (A) {$\mathcal{O}_{f}$};
       \node (B)[right of=A]{$\mathcal{O}_{f}\times *$};
       \node (C)[right of=B]{$\mathcal{O}_{f}\times \mathcal{O}_{1_n}$};
       \node (D)[below of=C]{$\mathcal{O}_{f}$};
       \draw[->] (A) to node[above=3]{$\cong$}(B);
       \draw[->] (B) to node[above=3]{$1\times\eta_{n}$}(C);
       \draw[->](C) to node[right=3]{$\Gamma$}(D);
       \draw[->](A) to node[below=3]{$1$} (D);
    \end{tikzpicture}
\]
\end{itemize}
\end{definition}

Given a pro $P$, we write $(\mathcal{O},\eta,\Gamma)$ to denote a deformation of $P$.

\begin{definition}
Let $P$ be a pro. A map of deformations of $P$
\[
\alpha:(\mathcal{O},\eta,\Gamma)\to (\mathcal{O}',\eta',\Gamma')
\]
consists of a morphism
\[
\alpha_{f}:\mathcal{O}_{f}\to \mathcal{O}'_{f}
\]
for all $n,k\geq 0$ and $f\in P(n,k)$, This satisfies the axioms for preservation of units and preservation of composition.
\end{definition}

Let $P$ be a pro. We write $\mathbf{Def}^P_\mathcal{V}$ to denote the category of deformations of $P$ in $\mathcal{V}$.

\begin{proposition}
Let $P$ be a pro. Then there is an induced functor 
\[
\Xi^P:\mathbf{Def}^P_\mathcal{V}\to \mathcal{V}\mathbf{Pro}.
\]
\end{proposition}

\begin{proof}
Let $\mathcal{O}:=(\mathcal{O},\eta,\Gamma)$ be a deformation of $P$. Define 
\[
\overline{\mathcal{O}}(n,m)=\coprod_{f\in P(n,m)}\mathcal{O}_f
\]
for all $n,m\geq 0$ and $1_k:*\to\overline{\mathcal{O}}$ as the following composite for $k\geq 0$, where $\iota_{\text{id}_k}$ is the structure morphism for the coproduct.
\[
*\xrightarrow{\eta^k}\mathcal{O}_{\text{id}_k}\xrightarrow{\iota_{\text{id}_k}}\overline{\mathcal{O}}(k,k)
\]
Let $t\geq 1$ and $n_1,\dots, n_t,m_1,\dots, m_t,k\geq 0$. Define 
\[
\Gamma:\overline{\mathcal{O}}(\sum_{i=1}^tm_i,k)\times \prod_{i=1}^t\overline{\mathcal{O}}(n_i,m_i)\to \overline{\mathcal{O}}(\sum_{i=1}^tn_i,k)
\]
as the following composite, where the first map is the identity  by definition, the second map is the obtained by using that $\mathcal{V}$ is a cartesian closed category, the third map is induced by the composition maps for the deformation, and the final map is induced by coproduct structure maps.
\[
\begin{tikzpicture}[node distance=2cm]
    \node (A) {$\overline{\mathcal{O}}(\sum_{i=1}^tm_i,k)\times \prod_{i=1}^t\overline{\mathcal{O}}(n_i,m_i)$};
    \node (B)[below of=A]{$\coprod_{g\in P(\sum_{i=1}^tm_i,k)}\mathcal{O}_g\times \prod_{i=1}^t\coprod_{f_i\in P(n_i,m_i)}\mathcal{O}_{f_i}$};
    \node (C) [below of=B]{$\coprod_{g\in P(\sum_{i=1}^tm_i,k)}\coprod_{f_1\in P(n_1,m_1)}\dots \coprod_{f_t\in P(n_t,m_t)}(\mathcal{O}_g\times\prod_{i=1}^t\mathcal{O}_{f_i})$};
    \node (D)[below of=C]{$\coprod_{g\in P(\sum_{i=1}^tm_i,k)}\coprod_{f_1\in P(n_1,m_1)}\dots \coprod_{f_t\in P(n_t,m_t)}\mathcal{O}_{g\circ\sum_{i=1}^tf_i}$};
    \node (E)[below of=D]{$\overline{\mathcal{O}}(\sum_{i=1}^tn_i,k)$};
    \draw[->](A) to node[left=3]{$=$}(B);
    \draw[->](B) to node[left=3]{$\cong$}(C);
    \draw[->](C) to node[left=3]{$\coprod_{g\in P(\sum_{i=1}^tm_i,k)}\coprod_{f_1\in P(n_1,m_1)}\dots \coprod_{f_t\in P(n_t,m_t)}\Gamma$}(D);
    \draw[->](D) to node[left=3]{$\iota$}(E);
\end{tikzpicture}
\]
We now define $\Xi^P(\mathcal{O})=(\overline{\mathcal{O}},1,\Gamma)$, which is well-defined by the axioms for a deformation. This provides the object part to the functor $\Xi^P$. The morphism part is defined in the same vein.
\end{proof}

\begin{definition}
    Let $P$ be a pro, $(\mathcal{O},\eta,\Gamma)$ be an deformation of $P$, and $C$ be a $\mathcal{V}$-category. An \emph{algebra of $(\mathcal{O},\eta,\Gamma)$} consists of an algebra of $\Xi^P(\mathcal{O})$ in $C$.
\end{definition}

\begin{proposition}
Let $P$ be a pro and $(\mathcal{O},\eta,\Gamma)$ be an deformation of $P$.
Then providing an algebra of $(\mathcal{O},\eta,\Gamma)$ in $\underline{\mathcal{V}}$ is equivalent to giving
\begin{itemize}
    \item an object $v$ of $\mathcal{V}$ and
    \item a morphism 
    \[
    \gamma:\mathcal{O}_{f}\times v^{\times n}\to v^{\times k},
    \]
    for all $n,k\geq 0$ and $f\in P(n,k)$
\end{itemize}
subject to the following axioms.
\begin{itemize}
\item \emph{Algebra-Composition Compatibility.} The following diagram commutes for all 
compatible indices.
\[
  \begin{tikzpicture}[node distance=5cm, auto]
        \node (A) {$\mathcal{O}_{f}\times\prod_{i=1}^t(\mathcal{O}_{g_i}\times v^{n_i})$};
        \node (B)[right of=A,node distance=1.25cm]{};
        \node(C)[above of=B,node distance=1.25cm]{$\mathcal{O}_{f}\times\prod_{i=1}^t\mathcal{O}_{g_i}\times v^{\sum_{i=1}^tn_i}$};
        \node(D)[below of=B,node distance=1.25cm]{$\mathcal{O}_{f}\times\prod_{i=1}^t v^{m_i}$};
        \node (E)[right of=C]{$\mathcal{O}_{f\circ\sum_{i=1}^tg_i}\times v^{\sum_{i=1}^tn_i}$};
        \node (F)[right of=D]{$\mathcal{O}_{f}\times v^{\sum_{i=1}^tm_i}$};
        \node (G)[right of=B,node distance=6.25cm]{$X(k)$};
        \draw[->] (A) to node[left=3] {$\cong$} (C);
        \draw[->](C) to node[above=3]{$\Gamma\times 1$}(E);
        \draw[->](E) to node[right=3]{$\gamma$}(G);
        \draw[->] (A) to node[left=3] {$1\times\prod_{i=1}^t\gamma$} (D);
        \draw[->] (D) to node[below=3]{$\cong$}(F);
        \draw[->](F) to node[right=3]{$\gamma$}(G);
    \end{tikzpicture}
\]
\item \emph{Algebra-Unity Compatibility.} The following diagram commutes for all 
compatible indices.
\[
  \begin{tikzpicture}[node distance=2.5cm, auto]
        \node (A) {$v^n$};
       \node (B)[right of=A]{$*\times v^n$};
       \node (C)[right of=B]{$\mathcal{O}_{1_n}\times v^n$};
       \node (D)[below of=C]{$v^n$};
       \draw[->] (A) to node[above=3]{$\cong$}(B);
       \draw[->] (B) to node[above=3]{$\eta_{n}\times 1$}(C);
       \draw[->](C) to node[right=3]{$\gamma$}(D);
       \draw[->](A) to node[below=3]{$1$} (D);
    \end{tikzpicture}
\]
\end{itemize}
\end{proposition}

\begin{proof}
We just read the data obtained by utilizing the cartesian closed structure of $\mathcal{V}$ and the fact that the $\mathcal{V}$-pro $\Xi^P(\mathcal{O},\eta,\Gamma)$ is completely determined by the deformation  $(\mathcal{O},\eta,\Gamma)$.
\end{proof}


\begin{notation}
 Let $P$ be a pro, $\mathcal{O}:=(\mathcal{O},\eta,\Gamma)$ be a deformation of $P$, and $C$ be a $\mathcal{V}$-category. We write $\mathbf{Alg}(\mathcal{O},C):=\mathbf{Alg}(\Xi^P(\mathcal{O}),C)$ to denote the category of models of $\mathcal{O}$ in $C$. 
\end{notation}

\subsection*{Deformations of controlled theories}
\begin{definition}
Let
\[
\Omega=\langle \mathcal{G}:P\xrightarrow{i}\mathbf{Fr}(\mathcal{G})\xrightarrow{\mathbf{st}}L\rangle
\]
be a controlled theory. A \emph{deformation} of $\Omega$ in $\mathcal{V}$ consists of a deformation $(\mathcal{O},\eta,\Gamma)$ of $\overline{\Omega}$ in $\mathcal{V}$ together with a morphism 
\[
W_f:*\to \mathcal{O}_{[f]}
\]
for all $n,k\geq 0$ and $f\in P(n,k)$ such that $W_{id_k}=\eta_k$ for all $k\geq 0$. This data satisfies the following axiom: given $t\geq 1$,  $n_1,\dots, n_t,m_1,\dots, m_t,k\geq 0$, $f_i\in P(n_i,m_i)$ for $1\leq i\leq t$, and $g\in P(\sum_{i=1}^tm_i,k)$, the following diagram commutes.
\[
  \begin{tikzpicture}[node distance=3cm, auto]
        \node (A) {$*$};
       \node (B)[right of=A]{$*\times\prod_{i=1}^t*$};
       \node (C)[right of=B,node distance=5cm]{$\mathcal{O}_{[g]}\times\prod_{i=1}^t\mathcal{O}_{[f_i]}$};
       \node (D)[below of=C]{$\mathcal{O}_{[g\circ \sum_{i=1}^tf_i]}$};
       \draw[->] (A) to node[above=3]{$\cong$}(B);
       \draw[->] (B) to node[above=3]{$W_{g}\times \prod_{i=1}^tW_{f_i}$}(C);
       \draw[->](C) to node[right=3]{$\Gamma$}(D);
       \draw[->](A) to node[below=3]{$W_{g\circ \sum_{i=1}^tf_i}$} (D);
    \end{tikzpicture}
\]
We write $(\mathcal{O},\eta,\Gamma, W)$ to denote a deformation of $\Omega$. 
\end{definition}


\begin{definition}
    Let
\[
\Omega=\langle \mathcal{G}:P\xrightarrow{i}\mathbf{Fr}(\mathcal{G})\xrightarrow{\mathbf{st}}L\rangle
\]
be a controlled theory and $(\mathcal{O},\eta,\Gamma, W)$ and $(\mathcal{O}',\eta',\Gamma',W')$ be deformations of $\Omega$. A map 
\[
\alpha:(\mathcal{O},\eta,\Gamma,W)\to (\mathcal{O}',\eta',\Gamma',W')
\]
of deformations of $\Omega$ consists of a map $\alpha:(\mathcal{O},\eta,\Gamma)\to (\mathcal{O}',\eta',\Gamma')$ of deformations over $\overline{\Omega}$ such that the following diagram commutes for all $n,k\geq 0 $ and $f\in P(n,k)$.
\[
  \begin{tikzpicture}[node distance=4cm, auto]
        \node (A) {$*$};
       \node (C)[right of=A]{$\mathcal{O}_{[f]}$};
       \node (D)[below of=C]{$\mathcal{O}'_{[f]}$};
       \draw[->] (A) to node[above=3]{$W_f$}(C);
       \draw[->](C) to node[right=3]{$\alpha$}(D);
       \draw[->](A) to node[below=3]{$W'_f$} (D);
    \end{tikzpicture}
\]
\end{definition}

\begin{notation}
Let
\[
\Omega=\langle \mathcal{G}:P\xrightarrow{i}\mathbf{Fr}(\mathcal{G})\xrightarrow{\mathbf{st}}L\rangle
\]
be a controlled theory. We write $\mathbf{Def}^{\Omega}_\mathcal{V}$ to denote the category of deformations of $\Omega$ in $\mathcal{V}$.
\end{notation}

There is a natural functor $i:\mathbf{Def}^{\Omega}_\mathcal{V}\to \mathbf{Def}^{\overline{\Omega}}_\mathcal{V}$ by definition of deformation of $\Omega$. Moreover, we obtain a functor 
\[
I^{(-)}:\mathbf{Def}^{\Omega}_\mathcal{V}\to *_{D(P)}\downarrow \text{id}_{\mathcal{V}\mathbf{Pro}}
\]
defined where we send $\mathcal{O}$ to the morphism 
\[
I^\mathcal{O}:D(P)\to \Xi^{\overline{\Omega}}(i(\mathcal{O}))
\]
of $\mathcal{V}$-pros induced by the $W$-structure maps. Given a deformation $\mathcal{O}$, we obtain an induced map of Lawvere $\mathcal{V}$-theories denoted by the following, where $L_P$ is the induced Lawvere theory from the pro $P$.
   \[
   D(L_P)\xrightarrow{J^\mathcal{O}}L^\mathcal{V}_{\Xi^{\overline{\Omega}}(i(\mathcal{O}))}
   \]
   There is an induced map of Lawvere theories 
   \[
   L_P\xrightarrow{\overline{i}}\mathbf{Fr}(\mathcal{G})
   \]
   induced by the faithful map of pros $i:P\to\mathbf{Fr}(\mathcal{G})$ which induces the following map of Lawvere $\mathcal{V}$-theories.
   \[
   D(L_P)\xrightarrow{D(\overline{i})}D(\mathbf{Fr}(\mathcal{G}))
   \]
   We now define the induced Lawvere $\mathcal{V}$-theory of $\mathcal{O}$, denoted by $L_{\mathcal{O}}$, as the pushout of the following span.
   \[
D(\mathbf{Fr}(\mathcal{G}))\xleftarrow{D(\overline{i})}D(L_P)\xrightarrow{J^\mathcal{O}}L^\mathcal{V}_{\Xi^{\overline{\Omega}}(i(\mathcal{O}))}
   \]
This provides us the object part of a functor $\Xi^\Omega:\mathbf{Def}^{\Omega}_\mathcal{V}\to *_{D(\mathbf{Fr}(\mathcal{G}))}\downarrow \text{id}_{\mathcal{V}\mathbf{Law}}$, where 
\[
\Xi^\Omega(\mathcal{O}):D(\mathbf{Fr}(\mathcal{G}))\to L_{\mathcal{O}}
\]
is the structure map for the pushout. When we post-compose with the forgetful functor 
\[
U:*_{D(\mathbf{Fr}(\mathcal{G}))}\downarrow \text{id}_{\mathcal{V}\mathbf{Law}}\to \mathcal{V}\mathbf{Law},
\]
we obtain a functor $L_{(-)}:\mathbf{Def}^{\Omega}_\mathcal{V}\to \mathcal{V}\mathbf{Law}$. 

\begin{definition}
Let 
\[
\Omega=\langle \mathcal{G}:P\xrightarrow{i}\mathbf{Fr}(\mathcal{G})\xrightarrow{\mathbf{st}}L\rangle
\]
be a controlled theory, $\mathcal{O}$ be a deformation of $\Omega$, and $C$ be a $\mathcal{V}$-category. Then an \emph{algebra} of $\mathcal{O}$ in $C$ is an algebra of $L_{\mathcal{O}}$ in $C$.
\end{definition}

We write $\mathbf{Alg}(\mathcal{O},C):=\mathbf{Alg}(L_\mathcal{O},C)$ to denote the category of models of $\mathcal{O}$ in $C$. There is naturally an induced functor $\Xi^\Omega(\mathcal{O})_*:\mathbf{Alg}(\mathcal{O},C)\to\mathbf{Alg}(D(\mathbf{Fr}(\mathcal{G})),C)$. This allows us to speak about 

\section{Categorification and homotopification}
\subsection*{Augmentations} We now provide a generalization of the augmentations of May in \cite{May2} and Schw\"anzl and Vogt in \cite{SchwanzlVogt1989} in the setting of deformation of controlled theories in cartesian closed model categories. We define notions of lax, pseudo, and strict augmentations.
\begin{definition}
Let $C$ be a cartesian closed model category and 
\[
\Omega=\langle \mathcal{G}:P\xrightarrow{i}\mathbf{Fr}(\mathcal{G})\xrightarrow{\mathbf{st}}L\rangle
\]
be a controlled theory. A \emph{lax augmentation}  is a deformation $(\mathcal{O},\eta,\Gamma,W)$ such that $\pi_0(W_f)=\text{id}_*$ for all $n,k\geq 0$ and $f\in P(n,k)$.
\end{definition}

\begin{example}
We may freely generate a deformation over the controlled theory $\Omega_{mon}$ of Example \ref{cont_theory_for_mon_gpds} in $\mathbf{Cat}$ as follows. For $n,k\geq 0$ and $f\in\overline{\Omega_{mon}}(n,k)$, we let $\mathbf{ob}(\mathcal{O}^{sk}_{[f]})=[f]$. We now give $\mathbf{ob}(\mathcal{O}^{sk}_{[f]})=[f]$ a poset structure for all $n,k\geq 0$ and $f\in\overline{\Omega_{mon}}(n,k)$. The morphisms of all the $\mathcal{O}^{sk}_{[f]}$ are generated by requiring the following inequalities.
\[
m\circ (m+1)\leq m\circ (1+m), m\circ (\eta+1)\leq \text{id}_1\leq m\circ (1+\eta)
\]
The structure maps $\eta^{sk},\Gamma^{sk},\text{ and } W^{sk}$ are defined and the axioms hold due to the pro structure and the thin structure on the components. A quick check shows that $(\mathcal{O}^{sk},\eta^{sk},\Gamma^{sk},W^{sk})$ is a lax augmentation.
 The category of algebras $\mathbf{Alg}^\Omega(\mathcal{O}^{sk},\eta^{sk},\Gamma^{sk},W^{sk})$ is isomorphic to the category of left skew monoidal categories (see Section 2 of \cite{Uus}) together with strict maps.
\end{example}

\begin{definition}
Let $C$ be a cartesian closed model category and 
\[
\Omega=\langle \mathcal{G}:P\xrightarrow{i}\mathbf{Fr}(\mathcal{G})\xrightarrow{\mathbf{st}}L\rangle
\]
be a controlled theory. A \emph{strong augmentation} is a deformation $(\mathcal{O},\eta,\Gamma,W)$ such that the map $W_f:*\to \mathcal{O}_{[f]}$ is a weak equivalence for all $n,k\geq 0$ and $f\in P(n,k)$.
\end{definition}
 
\begin{definition}
Let $C$ be a cartesian closed model category and 
\[
\Omega=\langle \mathcal{G}:P\xrightarrow{i}\mathbf{Fr}(\mathcal{G})\xrightarrow{\mathbf{st}}L\rangle
\]
be a controlled theory. A \emph{strict augmentation} is a deformation $(\mathcal{O},\eta,\Gamma)$ such that the terminal map $\mathcal{O}_{[f]}\xrightarrow{!}*$ is an isomorphism for all $n,k\geq 0$ and $f\in P(n,k)$.
\end{definition}
\begin{proposition}
    Let $C$ be a cartesian closed model category and 
\[
\Omega=\langle \mathcal{G}:P\xrightarrow{i}\mathbf{Fr}(\mathcal{G})\xrightarrow{\mathbf{st}}L\rangle
\]
be a controlled theory. There is only one strict algebraic augmentation, up to isomorphism. Moreover, its algebras are equivalent to the algebras over the $\mathcal{V}$-pro $D(\overline{\Omega})$.
\end{proposition}

\begin{proof}
The cartesian closed structure forces the structure to exist. The definition of a strict augmentation forces its uniqueness.
\end{proof}

\subsection{Omega-free deformations}
We point the reader to Construction \ref{pullback_pro} here so that the reader can reference what $\Omega^*$ is.

\begin{definition}
Let $C$ be a cartesian closed model category,
\[
\Omega=\langle \mathcal{G}:P\xrightarrow{i}\mathbf{Fr}(\mathcal{G})\xrightarrow{\mathbf{st}}L\rangle
\]
be a controlled theory, and $\mathcal{O}:=(\mathcal{O},\eta,\Gamma,W)$ be a deformation of $\Omega$. We say that $\mathcal{O}$ is \emph{$\Omega^*$-free} if given $n\geq 0$ and an invertible morphism $f\in\Omega^*(n,n)$, and if for all $k\geq 0$ and $g\in P(n,k)$ with $[g\circ f]=[g]$ in $\overline{\Omega}$, any diagram of the following form commutes precisely when $f=\text{id}_n$.

\begin{equation}\label{equat_for_freeness}
   \begin{tikzpicture}[node distance=2cm, auto]
        \node (A) {$\mathcal{O}_{[g]}$};
       \node (B)[right of=A]{$\mathcal{O}_{[g]}\times *$};
       \node (C)[right of=B,node distance=5cm]{$\mathcal{O}_{[g]}\times \mathcal{O}_{[f]}$};
       \node (D)[below of=C]{$\mathcal{O}_{[g\circ f]}$};
       \draw[->] (A) to node[above=3]{$\cong$}(B);
       \draw[->] (B) to node[above=3]{$1\times W_{f}$}(C);
       \draw[->](C) to node[right=3]{$\Gamma$}(D);
       \draw[->](A) to node[below=3]{$1$} (D);
    \end{tikzpicture} 
\end{equation}
\end{definition}

This provides an appropriate generalization for $\Sigma$-free operads in our setting.

\begin{example}
Let $\mathcal{O}$ be a $\Sigma$-free symmetric operad in a cartesian closed model category $\mathcal{V}$. Then $\mathcal{O}$ will induce a $\Omega^*$-free deformation of $\Omega_{cm}$ whose algebras are equivalent to the algebras of the operad.
\end{example}

\subsection*{The algebraic augmentation and realization}

\begin{construction}\label{constr_of_algebraic_aug}
Let $\Omega:=\langle \mathcal{G}:P\xrightarrow{i}\mathbf{Fr}(\mathcal{G})\xrightarrow{\mathbf{st}}L\rangle$ be a controlled theory. For all $n,k\geq 0$ and $[f]\in\overline{\Omega}(n,k)$, define $\mathcal{O}^\Omega_{[f]}=E([f])$ where $[f]$ is the set of operations identified under the reduction and $E$ is the right adjoint to the set of objects functor $\mathbf{ob}:\mathbf{Cat}\to\mathbf{Set}$. For all $k\geq 0$, we define a \emph{unit} functor 
\[
\eta_k^{\Omega}:*\to \mathcal{O}^\Omega_{[\text{id}_k]}
\]
that picks $\text{id}_k$. Define the \emph{composition} functor 
    \[
\Gamma^\Omega:\mathcal{O}^\Omega_{[g]}\times\prod_{i=1}^t\mathcal{O}^\Omega_{[f_i]}\to\mathcal{O}^\Omega_{[g\circ\sum_{i=1}^tf_i]},
    \]
    by utilizing the pro structure and the unique struture obtained by applying the functor $E$ for all $t\geq 1$, $n_1,\dots,n_t,m_1,\dots,m_t,k\geq 0$, $[f_i]\in\overline{\Omega}(n_i,m_i)$ for $1\leq i\leq t$, and $[g]\in\overline{\Omega}(\sum_{i=1}^tm_i,k)$. Define a choice functor $W^\Omega_f:*\to \mathcal{O}^\Omega_{[f]}$ for all $n,k\geq 0$ and $f\in P(n,k)$ such that 
    \[
    W^\Omega_{f}(*)=f
    \]
    for all $k\geq 0$.
This data satisfies the conditions of being a $\Omega^*$-free strong augmentation over $\Omega$ in $\mathbf{Cat}$. We call $\mathcal{AA}^\Omega:=(\mathcal{O}^\Omega,\Gamma^\Omega,\eta^\Omega,W^\Omega)$ the \emph{algebraic augmentation} of $\Omega$.
 \end{construction}

 This construction works equally as well for the category of groupoids since the image of $E$ lands in groupoids.

 \begin{theorem}\label{proof_of_freeness_of_alg_aug}
Let $\Omega:=\langle \mathcal{G}:P\xrightarrow{i}\mathbf{Fr}(\mathcal{G})\xrightarrow{\mathbf{st}}L\rangle$ be a controlled theory. The algebraic augmentation of $\Omega$ is a $\Omega^*$-free strong augmentation of $\Omega$.
 \end{theorem}

 \begin{proof}
     By construction, the algebraic augmentation $\mathcal{O}^\Omega$ is a strong augmentation. Let $n,k\geq 0$, $f\in\Omega^*(n,n)$ be invertible, and $g\in P(n,k)$ such that $[g\circ f]=[g]$. Notice that the diagram of \ref{equat_for_freeness} holds if and only if $f$ is the identity. Therefore $\mathcal{O}^\Omega$ is $\Omega^*$-free.
 \end{proof}

Now define a functor $\mathcal{AR}_{(-)}:\mathbf{cTh}\to 2\mathbf{Law}$ on objects by
\[
\mathcal{AR}_{\Omega}:=L_{\mathcal{O}^\Omega}.
\]
Now let 
\[
	(I,h,H):\langle \mathcal{G}:P\xrightarrow{i}\mathbf{Fr}(\mathcal{G})\xrightarrow{\mathbf{st}}L\rangle\to \langle \mathcal{G}':P'\xrightarrow{i'}\mathbf{Fr}(\mathcal{G}')\xrightarrow{\mathbf{st}'}L'\rangle
    \] 
be a map of controlled theories. To reduce writing, let $\Omega$ denote the first controlled theory, $\Omega'$ denote the second one, and $\Phi:=(I,h,H)$. For all $n,k\geq 0$ and $[f]\in\overline{\Omega}$, there is an induced functor
\[
\kappa_{[f]}:\mathcal{O}^{\Omega}_{[f]}\to \mathcal{O}^{\Omega'}_{[h(f)]}.
\]
Moreover, these functors respect the unities, the composition and the choices of operations. Therefore, we induce a map 
\[
\kappa(\Phi):\Xi^{\overline{\Omega}}(i(\mathcal{O}^{\Omega}))\to \Xi^{\overline{\Omega'}}(i(\mathcal{O}^{\Omega'}))
\]
 of $\mathbf{Cat}$-pros such that
 \[
 \kappa(\Phi)\circ I^{\mathcal{O}^\Omega}=I^{\mathcal{O}^{\Omega'}}\circ D(h).
 \]
 It follows from the universal property of the pushout, there is a unique map $\mathcal{AR}_{\Phi}:\mathcal{AR}_{\Omega}\to \mathcal{AR}_{\Omega'}$ of Lawvere $2$-theories such that the appropriate cube of Lawvere $2$-theories commutes. This provides the morphism part of our functor. We call $\mathcal{AR}_{(-)}$ the \emph{algebraic realization functor}. We may upgrade the algebraic realization functor to a functor
\[
\Xi:\mathbf{cTh}\to D\circ \mathbf{Fr}\downarrow\text{id}_{2\mathbf{Law}}
\]
defined on object by
\[
\Xi(\langle \mathcal{G}:P\xrightarrow{i}\mathbf{Fr}(\mathcal{G})\xrightarrow{\mathbf{st}}L\rangle)=(\mathcal{G}, \mathcal{AR}_{\Omega}, \Xi^\Omega(\mathcal{O}^\Omega)).
\]

Let us now explain the naming convention of algebraic augmentation/realization. Given a controlled theory $\Omega=\langle \mathcal{G}:P\xrightarrow{i}\mathbf{Fr}(\mathcal{G})\xrightarrow{\mathbf{st}}L\rangle$, there is a quotient map $\mathbf{st}:\mathcal{AR}_{\Omega}\to D(L)$ of Lawvere $2$-theories that we can view as an augmentation in the sense of \cite{SchwanzlVogt1989}.

\begin{example}
Let $\Omega_{mon}$ be the controlled theory of Example \ref{cont_theory_for_mon_gpds}. The algebraic realization of $\Omega_{mon}$ is a Lawvere $2$-theory whose $2$-category of models consists of monoidal categories, strict monoidal functors, and monoidal transformations. If we use the groupoid enrichment, we obtain monoidal groupoids.
\end{example}

\begin{example}
Let $\Omega_{cm}$ be the controlled theory of Example \ref{cont_theory_for_sym_mon_gpds}. The algebraic realization of $\Omega_{cm}$ is a Lawvere $2$-theory whose $2$-category of models consists of symmetric monoidal categories, strict monoidal functors, and monoidal transformations. If we use the groupoid enrichment, we obtain symmetric monoidal groupoids.
\end{example}

The previous two examples are due to Mac Lane in \cite{MacLane1963}. The coherence structure is completely dependent on the induced maps of pros, where $\mathbf{M}$ is the pro for monoids and $\mathbf{CM}$ is the prop for commutative monoids. 
\[
\mathbf{st}:\mathbf{N}[\mathcal{M}]\to \mathbf{M}
\]
\[
\mathbf{st}:\Sigma(\mathcal{M})\to\mathbf{CM}
\]

\begin{example}
Let $\Omega_{grp}$ be the controlled theory of Example \ref{cont_theory_for_2_grps}. The algebraic realization of $\Omega_{grp}$ is a Lawvere $2$-theory whose $2$-category of models consists of coherent group-like monoidal categories, strict monoidal functors, and monoidal transformations. If we use the groupoid enrichment, we obtain coherent $2$-groups.
\end{example}

This example follows from Theorem 10.3.33 of \cite{Parab2025}.  Also, we may recapture this example by showing that the induced augmentation $\mathbf{st}:\mathcal{AR}_{\Omega_{grp}}\to D(\mathbf{G})$ is a bi-equivalence of Lawvere $2$-theories, where $\mathbf{G}$ is the Lawvere theory for groups.

We now extend our definition of the algebraic realization for a controlled theory to a definition for the algebraic realization of a connected diagram of controlled theories. Let $J:I\to \mathbf{cTh}$ be a connected diagram of controlled theories. We define the algebraic realization of the diagram $D$ to be 
\[
\mathcal{AR}_D:=\text{colim}(\mathcal{AR}_{(-)}\circ J)
\]
This provides us a functor $\mathcal{AR}_{(-)}:\mathbf{Diag}^+(\mathbf{cTh})\to 2\mathbf{Law}$. This functor definition is important because the functor 
$\mathcal{AR}_{(-)}:\mathbf{cTh}\to 2\mathbf{Law}$
does not preserve connected colimits.

\begin{example}\label{Pic_Ex}
Let $\Omega_{pic}$ be the connected diagram of controlled theories of Example \ref{conn_diag_cont_theory_for_Pic}. The algebraic realization of $\Omega_{pic}$ is a Lawvere $2$-theory whose $2$-category of models consists of group-like symmetric monoidal categories, strict monoidal functors, and monoidal transformations. If we use the groupoid enrichment, we obtain Picard groupoids.
\end{example}

We now focus in on Example \ref{Pic_Ex} to understand why we avoid the problem of \cite{SchwanzlVogt1989} in our setting. We obtain an induced map 
\[
\mathbf{st}:\mathcal{AR}_{\Omega_{pic}}\to D(\mathbf{A}),
\]
where $\mathbf{A}$ is the Lawvere theory for abelian groups.
We now consider the following lemma.

\begin{proposition}\label{non_contract_lemma}
The groupoid $\mathbf{st}^{-1}(\text{id}_1)$ is non-contractible.
\end{proposition}

\begin{proof}
    The operation $m\circ (m+1)\circ (\tau+1)\circ(1+1+\iota)\circ\Delta_3$ reduces as
    \[
    m\circ (m+1)\circ (\tau+1)\circ(1+1+\iota)\circ\Delta_3=m\circ (m+1)\circ (1+1+\iota)\circ(\tau+1)\circ\Delta_3
    \]
    \[
    =m\circ (m+1)\circ (1+1+\iota)\circ\Delta_3.
    \]
    There is a chain of isomorphisms
    \[
    m\circ (m+1)\circ (1+1+\iota)\circ\Delta_3\sim m\circ (1+m)\circ (1+1+\iota)\circ\Delta_3
    \]
    \[
    \sim m\circ (1+[m\circ (1+\iota)\circ\Delta_2])\circ\Delta_2\sim m\circ(1+\eta\circ !)\circ\Delta_2
    \]
    \[
    = m\circ (1+\eta)\circ (1+!)\circ\Delta_2=m\circ(1+\eta)\sim\text{id}_1.
    \]
    Denote the composite as $b$. The chain above induces the following chain of isomorphisms.
    \[
    \text{id}_1\sim  m\circ (m+1)\circ (1+1+\iota)\circ\Delta_3\sim  m\circ (m+1)\circ(\tau+1)\circ (1+1+\iota)\circ\Delta_3
    \]
    \[
    =m\circ (m+1)\circ (1+1+\iota)\circ\Delta_3\sim\text{id}_1
    \]
    that we will label as $t$. This isomorphism $t$ is a non-trivial automorphism of $\text{id}_1$. This is verified by checking its action on all models. Therefore the groupoid $\mathbf{st}^{-1}(\text{id}_1)$ is non-contractible. 
\end{proof}
The morphism $t$ induces the map called the trace on models in \cite{Dugg,JoyalStreetVerity1996}. If we take a look at $t$, we have that $t\circ t=\text{id}_{\text{id}_1}$. Moreover, $t$ is the only non-trivial automorphism of $\text{id}_1$, so that
\[
\mathbf{st}^{-1}(\text{id}_1)(\text{id}_1,\text{id}_1)=\mathbf{Z}/2.
\]
We are forced to have that 
\[
\mathbf{st}^{-1}(\text{id}_1)(f,g)=\mathbf{Z}/2
\]
for all $f,g\in \mathbf{ob}(\mathbf{st}^{-1}(\text{id}_1))$ since $\mathbf{st}^{-1}(\text{id}_1)$ is connected. The morphism from $f$ to $g$ is determined by the notion of parity introduced by Dugger in \cite{Dugg}. In fact, we are interested in a proof of the coherence theorem of \cite{Dugg} whose argument works on the syntactic data presented here.

\subsection*{The nerve augmentation and realization}
We begin by writing $\mathbf{sSTh}$ for the category of Lawvere $\mathbf{sSet}$-theories to have clean notation. We may now define the nerve augmentation of a controlled theory as follows.
\begin{construction}
Let $\Omega:=\langle \mathcal{G}:P\xrightarrow{i}\mathbf{Fr}(\mathcal{G})\xrightarrow{\mathbf{st}}L\rangle$ be a controlled theory and $\mathcal{O}^\Omega$ be the algebraic augmentation of $\Omega$. Define
\[
\mathcal{NA}^\Omega
\]
to be the deformation of $\Omega$ in $\mathbf{sSet}$ obtained by applying the nerve functor $N:\mathbf{Cat}\to\mathbf{sSet}$ to all the structure data of the algebraic augmentation of $\Omega$. We call $\mathcal{NA}^\Omega$ the nerve augmentation of $\Omega$. This is a deformation of $\Omega$ since $N$ is a right adjoint and therefore preserves products.
\end{construction}

 \begin{theorem}
Let $\Omega:=\langle \mathcal{G}:P\xrightarrow{i}\mathbf{Fr}(\mathcal{G})\xrightarrow{\mathbf{st}}L\rangle$ be a controlled theory. The classifying augmentation of $\Omega$ is a $\Omega^*$-free strong augmentation of $\Omega$.
 \end{theorem}
 
 \begin{proof}
Since the algebraic augmentation was $\Omega^*$-free, so is the nerve augmentation.
 \end{proof}

The nerve functor preserves coproducts. Therefore we have no issue defining 
 \[
 \mathcal{NR}_{(-)}:\mathbf{cTh}\to\mathbf{sSet}\mathbf{Law}
 \]
 in an identical fashion to the algebraic realization from the algebraic augmentation. We call $\mathcal{NR}_{(-)}$ the \emph{nerve realization functor}. This further induces a functor 
 \[
 \mathcal{NR}_{(-)}:\mathbf{Diag}^+(\mathbf{cTh})\to\mathbf{sSet}\mathbf{Law}
 \]
the same way as before.

\subsection*{Applications to homotopy theory}
Before we begin this subsection, we note that the categories of algebras in this section have the projective model structure inherited from $\mathbf{sSet}$ by Theorem 7.2 of \cite{Rezk2002}. 
\begin{proposition}\label{equiv_to_A_inf_op}
Let $\Omega_{mon}$ be the controlled theory of Example \ref{cont_theory_for_mon_gpds}. There is an $A_\infty$-operad $\mathcal{A}$ together with an equivalence of categories.
\[
\mathbf{Alg}(\mathcal{NR}_{\Omega_{mon}},\underline{\mathbf{sSet}})\simeq \mathbf{Alg}(\mathcal{A})
\]
\end{proposition}
\begin{proof}
Define $\mathcal{A}_{n}=\coprod_{[f]\in \overline{\Omega_{mon}}(n,1)}\mathcal{N}\mathcal{O}^{\Omega_{mon}}_{[f]}$. As $\overline{\Omega_{mon}}(n,1) \cong *$, the map $W_{n}:\Delta^0\to\mathcal{A}_{n}$ is a weak equivalence for all $n\geq 0$. Therefore $\mathcal{A}$ is an $A_{\infty}$-operad. The $A_{\infty}$-operad $\mathcal{A}$ induces a Lawvere $\mathbf{sSet}$-theory $L_{\mathcal{A}}$ that is isomorphic to $\mathcal{NR}_{\Omega_{mon}}$. Therefore we induce our equivalence of categories as the composite of the following chain.
\[
\mathbf{Alg}(\mathcal{NR}_{\Omega_{mon}},\underline{\mathbf{sSet}})\cong \mathbf{Alg}(L_{\mathcal{A}},\underline{\mathbf{sSet}})\simeq\mathbf{Alg}(\mathcal{A})
\]
This completes our proof.
\end{proof}

We may upgrade our equivalence to a Quillen equivalence since the model structures on both categories of algebras agree. Therefore we face no harm in saying that the category of algebras of $\mathcal{NR}_{\Omega_{mon}}$ is the category of $A_\infty$-spaces on a homotopical and categorical level. Therefore, we write 
\[
A_\infty\text{-Spaces}:=\mathbf{Alg}(\mathcal{NR}_{\Omega_{mon}},\underline{\mathbf{sSet}}).
\]

Let
\[
\mathbf{st}:\mathcal{NR}_{\Omega_{mon}}\longrightarrow D(\mathbf{M})
\]
be the canonical morphism of theories. This morphism is a weak equivalence in
$\mathbf{sSTh}$. By Theorem~4.1 of \cite{BergerMoerdijk2007} (equivalently,
Corollary~8.6 of \cite{Rezk2002}), it induces the following Quillen equivalence.
\[
\begin{tikzcd}[column sep=large]
\mathcal{A}_{\infty}\text{-Spaces}
\arrow[r, shift left=1.2ex, "\mathbf{st}"]
&
\mathbf{sMon}
\arrow[l, shift left=1.2ex, "U"]
\end{tikzcd}
\]

We now apply the same construction to obtain a model for
$\infty$-groups. Define
\[
\mathcal{A}^{gl}_{\infty}\text{-Spaces}
:=
\mathbf{Alg}(\mathcal{NR}_{\Omega_{grp}},\underline{\mathbf{sSet}}).
\]
If $X$ is an object of $\mathcal{A}^{gl}_{\infty}\text{-Spaces}$, we call $X$ a \emph{coherent group-like $A_{\infty}$-space}. 

\begin{theorem}\label{theorem_sixpointtwenty}
The projective model structure on
$\mathcal{A}^{gl}_{\infty}\text{-Spaces}$
is a model for $\infty$-groups.
\end{theorem}

\begin{proof}
The canonical morphism of theories
\[
\mathbf{st}:
\mathcal{NR}_{\Omega_{grp}}
\longrightarrow
D(\mathbf{G})
\]
induces a Quillen adjunction
\[
\begin{tikzcd}[column sep=large]
\mathcal{A}^{gl}_{\infty}\text{-Spaces}
\arrow[r, shift left=1.2ex, "\mathbf{st}"]
&
\mathbf{sGrp}
\arrow[l, shift left=1.2ex, "U"]
\end{tikzcd}
\]
whose morphisms on hom-objects are weak equivalences of simplicial sets.
Corollary~8.6 of \cite{Rezk2002} therefore implies that this adjunction is a
Quillen equivalence. Since simplicial groups model connective homotopy types,
it follows that
$\mathcal{A}^{gl}_{\infty}\text{-Spaces}$
is a model category for $\infty$-groups.
\end{proof}

\begin{proposition}
Let $\Omega_{cm}$ be the controlled theory of Example \ref{cont_theory_for_sym_mon_gpds}. There is an $E_\infty$-operad $\mathcal{E}$ together with an equivalence of categories.
\[
\mathbf{Alg}(\mathcal{NR}_{\Omega_{cm}},\underline{\mathbf{sSet}})\simeq \mathbf{Alg}(\mathcal{E})
\]
\end{proposition}

\begin{proof}
Repeat the proof of Proposition \ref{equiv_to_A_inf_op} making sure to account for the symmetric group actions coming from the deformation structure over the controlled theory $\Omega_{cm}$.
\end{proof}

Once again, this allows us to write 
\[
E_\infty\text{-spaces}:=\mathbf{Alg}(\mathcal{NR}_{\Omega_{cm}},\underline{\mathbf{sSet}})
\]
without any worry. Now write
\[
\mathcal{E}^{gl}_{\infty}\text{-Spaces}:=\mathbf{Alg}(\mathcal{NR}_{\Omega_{pic}},\underline{\mathbf{sSet}}).
\]
We have a pullback diagram of categories.
\[
\begin{tikzpicture}[node distance=3cm]
\node (A) {$\mathcal{E}^{gl}_{\infty}\text{-Spaces}$};
\node (B)[right of=A] {$E_\infty\text{-spaces}$};
\node (C)[below of=A]{$\mathcal{A}^{gl}_{\infty}\text{-Spaces}$};
\node (D)[right of=C]{$A_\infty\text{-Spaces}$};
\draw[->] (A) to node[above=3] {$U$} (B);
\draw[->] (A) to node[left=3] {$U$} (C);
\draw[->] (B) to node[right=3] {$U$} (D);
\draw[->] (C) to node[below=3] {$U$} (D);
\end{tikzpicture}
\]
 If $X$ is an object of $\mathcal{E}^{gl}_{\infty}\text{-Spaces}$, we call $X$ a \emph{coherent group-like $E_\infty$-space}.

The pushout that this square comes from induces a map of theories
\[
\mathbf{st}:\mathcal{NR}_{\Omega_{pic}}\to D(\mathbf{A}),
\]
where $\mathbf{A}$ is the Lawvere theory for abelian groups. We may translate the results of Proposition \ref{non_contract_lemma} from $\mathbf{Cat}$ to $\mathbf{sSet}$

\begin{theorem}\label{not_we_eq}
The morphism of theories 
\[
\mathbf{st}:\mathcal{NR}_{\Omega_{pic}}\to D(\mathcal{A})
\]
is not a weak equivalence in $\mathbf{sSTh}$.
\end{theorem}

\begin{proof}
We will show that $\mathbf{st}^{-1}(\text{id}_1)$ is not contractible. Define $f:\partial\Delta^2\to \mathbf{st}^{-1}(\text{id}_1)$ by
\[
f_0(0)=\text{id}_1,
\]
\[
f_0(1)=f_0(2)=m\circ (m+1)\circ(1+1+\iota)\circ\Delta_3,
\]
\[
f_1(0\to 1)=(\text{id}_1\xrightarrow{b^{-1}}m\circ (m+1)\circ(1+1+\iota)\circ\Delta_3),
\]
\[
f_1(1\to 2)=(m\circ (m+1)\circ(1+1+\iota)\circ\Delta_3\xrightarrow{1_m\circ (\beta+1)\circ 1_{(1+1+\iota)\circ \Delta_3}}m\circ (m+1)\circ(1+1+\iota)\circ\Delta_3),
\]
and
\[
f_1(0\to 2)=(\text{id}_1\xrightarrow{b^{-1}}m\circ (m+1)\circ(1+1+\iota)\circ\Delta_3)
\]
where $b$ and $t$ are as they are in Proposition \ref{non_contract_lemma} and $\beta$ is the unique map from $m$ to $m\circ \tau$. The following diagram has no lift in $\mathbf{sSet}$.
\[
\begin{tikzpicture}[node distance=3cm]
\node (A) {$\partial\Delta^2$};
\node (B)[right of=A] {$\mathbf{st}^{-1}(\text{id}_1)$};
\node (C)[below of=A]{$\Delta^2$};
\node (D)[right of=C]{$*$};
\draw[->] (A) to node[above=3] {$f$} (B);
\draw[->] (A) to node[left=3] {$j_2$} (C);
\draw[->] (B) to node[right=3] {$!$} (D);
\draw[->] (C) to node[below=3] {$!$} (D);
\end{tikzpicture}
\]
Therefore $\mathbf{st}^{-1}(\text{id}_1)$ is not contractible. This means that $\mathbf{st}$ is not a weak equivalence in $\mathbf{sSTh}$.
\end{proof}

\begin{references*}
\bibitem[Ad\'{a}mek and Rosick\'{y}, 1994]{AdRo}
J.~Ad\'{a}mek and J.~Rosick\'{y},
\emph{Locally Presentable and Accessible Categories}.
Cambridge University Press, Cambridge, 1994.

\bibitem[Anderson, 1972]{And1}
D.~W. Anderson,
\emph{Fibrations and geometric realizations}.
Bull. Amer. Math. Soc. \textbf{78} (1972), no.~4, 521--525.

\bibitem[Ara, 2013]{Dim1}
D. Ara,
\emph{On the homotopy theory of Grothendieck $\infty$-groupoids}.
J. Pure Appl. Algebra \textbf{217} (2013), no.~7, 1237--1278.

\bibitem[Arkor and McDermott, 2025]{ArkorMcDermott2025RelativeMonadicity}
N.~Arkor and D.~McDermott,
\emph{Relative monadicity}.
Journal of Algebra \textbf{663} (2025), 399--434.

\bibitem[Arkor and McDermott, 2025]{ArkorMcDermott2025}
N.~Arkor and D.~McDermott,
\emph{The nerve theorem for relative monads}.
Theory and Applications of Categories \textbf{43} (2025), 403--454.

\bibitem[Baez and Williams, 2020]{BaezWilliams2020}
J.~C. Baez and C. Williams,
\emph{Enriched Lawvere Theories for Operational Semantics}.
In \emph{Proceedings of the Applied Category Theory 2019 Conference},
J.~C. Baez and B. Coecke (eds.),
Electron. Proc. Theor. Comput. Sci. \textbf{323} (2020), 106--135.

\bibitem[Barratt and Priddy, 1972]{BarrattPriddy1972}
M.~G. Barratt and S. Priddy,
\emph{On the homology of non-connected monoids and their associated groups}.
Comment. Math. Helv. \textbf{47} (1972), 1--14.

\bibitem[Badzioch, 2002]{bern1}
B. Badzioch,
\emph{Algebraic theories in homotopy theory}.
Ann. of Math. (2) \textbf{155} (2002), no.~3, 895--913.

\bibitem[Ba\v{s}i\'c and Nikolaus, 2014]{BasicNikolaus}
M. Ba\v{s}i\'c and T. Nikolaus,
\emph{Dendroidal sets as models for connective spectra}.
J. K-Theory \textbf{14} (2014), no.~3, 387--421.

\bibitem[Berger and Moerdijk, 2007]{BergerMoerdijk2007}
C. Berger and I. Moerdijk,
\emph{Resolution of Coloured Operads and Rectification of Homotopy Algebras}.
In \emph{Categories in Algebra, Geometry and Mathematical Physics},
Contemp. Math. \textbf{431}.
Amer. Math. Soc., Providence, RI, 2007, 31--58.

\bibitem[Bressie, 2020]{Bressie2020}
P.~M. Bressie,
\emph{The $\omega$-categorification of Algebraic Theories}.
Available at arXiv:2006.07191, 2020.

\bibitem[Cohen, 2009]{Cohen2009}
J.~A. Cohen,
\emph{Coherence for rewriting 2-theories}.
arXiv:0904.0125 [math.CT], 2009.

\bibitem[Dugger, 2014]{Dugg}
D. Dugger,
\emph{Coherence for invertible objects and multigraded homotopy rings}.
Algebr. Geom. Topol. \textbf{14} (2014), no.~2, 1055--1106.

\bibitem[Dwyer, Hirschhorn, and Kan, 1997]{DwyerHirschhornKan1997}
W.~G. Dwyer, P.~S. Hirschhorn, and D.~M. Kan,
\emph{Model Categories and More General Abstract Homotopy Theory}.
Unpublished manuscript, 1997.

\bibitem[Elmendorf and Mandell, 2006]{ELM2}
A.~D. Elmendorf and M.~A. Mandell,
\emph{Rings, Modules, and Algebras in Infinite Loop Space Theory}.
Adv. Math. \textbf{205} (2006), no.~1, 163--228.

\bibitem[Fong and Spivak, 2019]{FongSpivak2019}
B. Fong and D.~I. Spivak,
\emph{An Invitation to Applied Category Theory: Seven Sketches in Compositionality}.
Cambridge University Press, Cambridge, 2019.

\bibitem[Goerss and Jardine, 1999]{JarGo}
P.~G. Goerss and J.~F. Jardine,
\emph{Simplicial Homotopy Theory}.
Progress in Mathematics~174.
Birkh\"auser Verlag, Basel, 1999.

\bibitem[Goerss and Schemmerhorn, 2007]{GoerssSchemmerhorn}
P.~G. Goerss and K. Schemmerhorn,
\emph{Model Categories and Simplicial Methods}.
In \emph{Interactions Between Homotopy Theory and Algebra},
Contemp. Math. \textbf{436}.
Amer. Math. Soc., Providence, RI, 2007, 3--49.

\bibitem[Gould, 2008]{Gould2008}
M.~R. Gould,
\emph{Coherence for Categorified Operadic Theories}.
Ph.D. thesis,
University of Glasgow, 2008.

\bibitem[Gray, 1974]{Gray1974}
J.~W. Gray,
\emph{2-algebraic theories and triples}.
Cahiers de Topologie et Géométrie Différentielle \textbf{14} (1974), 178--180.

\bibitem[Grothendieck, 2022]{Gr1}
A. Grothendieck,
\emph{Pursuing Stacks (\`a la poursuite des champs). Vol.~I}.
Documents Math\'ematiques~20.
Soci\'et\'e Math\'ematique de France, Paris, 2022.

\bibitem[Gurski and Johnson, 2026]{GurskiJohnson2026}
N. Gurski and N. Johnson,
\emph{Invertibility and Parity in Symmetric Monoidal Categories}.
Appl. Categ. Structures \textbf{34} (2026), Article~34.

\bibitem[Gurski, Johnson, and Osorno, 2019]{GuJoOs}
N. Gurski, N. Johnson, and A.~M. Osorno,
\emph{The 2-dimensional stable homotopy hypothesis}.
J. Pure Appl. Algebra \textbf{223} (2019), no.~10, 4348--4383.

\bibitem[Henry and Lanari, 2023]{HenLan}
S. Henry and E. Lanari,
\emph{On the homotopy hypothesis for 3-groupoids}.
Theory Appl. Categ. \textbf{39} (2023), Paper No.~26, 735--768.

\bibitem[Hirschhorn, 2003]{Hirschhorn2003}
P.~S. Hirschhorn,
\emph{Model Categories and Their Localizations}.
Mathematical Surveys and Monographs~99.
American Mathematical Society, Providence, RI, 2003.

\bibitem[Hovey, 1998]{hov3}
M. Hovey,
\emph{Monoidal Model Categories}.
arXiv:math/9803002 [math.AT], 1998.

\bibitem[Johnson and Osorno, 2012]{OSNJ}
N. Johnson and A.~M. Osorno,
\emph{Modeling stable one-types}.
Theory Appl. Categ. \textbf{26} (2012), no.~20, 520--537.

\bibitem[Johnson and Yau, 2020]{JohnsonYau2020}
N. Johnson and D. Yau,
\emph{2-Dimensional Categories}.
Oxford University Press, Oxford, 2020.

\bibitem[Joyal, 2008]{Joyal1}
A. Joyal,
\emph{The Theory of Quasi-Categories and Its Applications}.
Lecture notes, 2008.
Available at \url{https://www.math.uchicago.edu/~may/IMA/Joyal.pdf}.

\bibitem[Joyal, Street, and Verity, 1996]{JoyalStreetVerity1996}
A. Joyal, R. Street, and D. Verity,
\emph{Traced Monoidal Categories}.
Math. Proc. Cambridge Philos. Soc. \textbf{119} (1996), no.~3, 447--468.

\bibitem[Kelly, 1982]{Kelly3}
G.~M. Kelly,
\emph{Basic Concepts of Enriched Category Theory}.
London Mathematical Society Lecture Note Series~64.
Cambridge University Press, Cambridge, 1982.

\bibitem[Kelly and Lack, 2001]{KellyLack2001}
G.~M. Kelly and S. Lack,
\emph{$V$-Cat is locally presentable or locally bounded if $V$ is so}.
Theory Appl. Categ. \textbf{8} (2001), 555--575.

\bibitem[Lanari, 2018]{LanE}
E. Lanari,
\emph{A semi-model structure for Grothendieck weak 3-groupoids}.
Preprint, 2018.

\bibitem[Lawvere, 1963]{Law2}
F.~W. Lawvere,
\emph{Functorial semantics of algebraic theories}.
Proc. Natl. Acad. Sci. USA \textbf{50} (1963), 869--872.

\bibitem[Lawvere, 2002]{Lawvere1973}
F.~W. Lawvere,
\emph{Metric Spaces, Generalized Logic, and Closed Categories}.
Repr. Theory Appl. Categ. \textbf{1} (2002), 1--37.

\bibitem[Leinster, 2004]{Lein}
T. Leinster,
\emph{Higher Operads, Higher Categories}.
London Mathematical Society Lecture Note Series~298.
Cambridge University Press, Cambridge, 2004.

\bibitem[Lurie, 2009]{Lurie2009}
J. Lurie,
\emph{Higher Topos Theory}.
Annals of Mathematics Studies~170.
Princeton University Press, Princeton, NJ, 2009.

\bibitem[Mac~Lane, 1963]{MacLane1963}
S. Mac~Lane,
\emph{Natural Associativity and Commutativity}.
Rice Univ. Stud. \textbf{49} (1963), no.~4, 28--46.

\bibitem[Maltsiniotis, 2005]{Geo2}
G. Maltsiniotis,
\emph{La th\'eorie de l'homotopie de Grothendieck}.
Ast\'erisque \textbf{301} (2005), vi+140 pp.

\bibitem[May, 1974]{Ma1}
J.~P. May,
\emph{$E_\infty$ spaces, group completions, and permutative categories}.
In \emph{New Developments in Topology},
London Mathematical Society Lecture Note Series~11.
Cambridge University Press, London, 1974, 61--94.

\bibitem[May, 1972]{May2}
J.~P. May,
\emph{The Geometry of Iterated Loop Spaces}.
Lecture Notes in Mathematics~271.
Springer-Verlag, Berlin, 1972.

\bibitem[May and Thomason, 1978]{Ma_Th}
J.~P. May and R. Thomason,
\emph{The uniqueness of infinite loop space machines}.
Topology \textbf{17} (1978), no.~3, 205--224.

\bibitem[McDermott and Uustalu, 2022]{McDermottUustalu2022}
D. McDermott and T. Uustalu,
\emph{What Makes a Strong Monad?}.
Electron. Proc. Theor. Comput. Sci. \textbf{360} (2022), 113--133.

\bibitem[Muro, 2015]{Muro2015}
F. Muro,
\emph{Dwyer--Kan homotopy theory of enriched categories}.
J. Topol. \textbf{8} (2015), no.~2, 377--413.

\bibitem[Parab, 2025]{Parab2025}
A. Parab,
\emph{Coherence and Symmetrization of Categorical Groups}.
Ph.D. thesis, The Ohio State University, Columbus, OH, 2025.

\bibitem[Patterson, 2023]{Patt1}
E. Patterson,
\emph{Unbiased monoidal categories are pseudo-elements}.
Topos Institute Blog, 2023.
Available at \url{https://topos.institute/blog/2023-08-15-unbiased-pseudomonoids/}.

\bibitem[Perutka, 2026]{Perutka2026}
T. Perutka,
\emph{$2$-Dimensional Lawvere Theories, Commutativity, and Higher Day Convolution}.
arXiv:2602.14332 [math.CT], 2026.

\bibitem[Power, 1999]{Power1999}
J.~Power,
\emph{Enriched Lawvere theories}.
Theory and Applications of Categories \textbf{6} (1999), No.~7, 83--93.

\bibitem[Quillen, 1967]{Quillen67}
D.~G. Quillen,
\emph{Homotopical Algebra}.
Lecture Notes in Mathematics~43.
Springer-Verlag, Berlin, 1967.

\bibitem[Riehl and Verity, 2020]{RiehlVerity2020}
E. Riehl and D. Verity,
\emph{Infinity Category Theory from Scratch}.
Higher Structures \textbf{4} (2020), no.~1, 115--167.

\bibitem[Rezk, 2002]{Rezk2002}
C. Rezk,
\emph{Every Homotopy Theory of Simplicial Algebras Admits a Proper Model}.
Topology Appl. \textbf{119} (2002), no.~1, 65--94.

\bibitem[Schw\"anzl and Vogt, 1989]{SchwanzlVogt1989}
R. Schw\"anzl and R.~M. Vogt,
\emph{$E_\infty$-monoids with coherent homotopy inverses are Abelian groups}.
Topology \textbf{28} (1989), no.~4, 481--484.

\bibitem[Str{\o}m, 1972]{Strom}
A. Str{\o}m,
\emph{The homotopy category is a homotopy category}.
Arch. Math. \textbf{23} (1972), 435--441.

\bibitem[Taylor, 2026]{Taylor2026}
J. Taylor,
\emph{Controlled Theories and Infinity Lawvere Theories: Weakening Axioms in the Globular Setting}.
Ph.D. thesis, Case Western Reserve University, Cleveland, OH, 2026.

\bibitem[Ugle\v{s}i\'c, 2019]{Uglesic2019}
N. Ugle\v{s}i\'c,
\emph{Enriched pro-categories and shapes}.
arXiv:1905.07181 [math.CT], 2019.

\bibitem[Uustalu, 2014]{Uus}
T. Uustalu,
\emph{Coherence for skew-monoidal categories}.
In \emph{Proceedings of the 5th Workshop on Mathematically Structured Functional Programming},
Electron. Proc. Theor. Comput. Sci. \textbf{153} (2014), 68--77.

\bibitem[Yanofsky, 2000]{yan1}
N.~S. Yanofsky,
\emph{Coherence, Homotopy and 2-Theories}.
arXiv:math/0007033 [math.CT], 2000.

\bibitem[Yau, 2016]{Yau2016}
D. Yau,
\emph{Colored Operads}.
Graduate Studies in Mathematics~170.
American Mathematical Society, Providence, RI, 2016.
\end{references*}

\end{document}